\documentclass[11pt,leqno]{amsart}
\usepackage{amscd}
\usepackage{amssymb}
\usepackage{amsfonts}
\usepackage[centertags]{amsmath}
\usepackage{latexsym}
\usepackage{verbatim}
\usepackage{amsthm}
\usepackage{color}
\usepackage{pgfplots}
\usepackage{tikz}
\usetikzlibrary{intersections}
\usepackage{amsfonts}
\usepackage{amssymb}
\usepackage{amsmath}
\usepackage{amsthm}
\usepackage[english]{babel}
\usepackage{subfigure}
\usepackage{enumerate}
\usepackage{epsfig}
\usepackage{graphicx}
\usepackage{color}
\usepackage{hyperref}
\pdfoutput=1

\newcommand{\NN}{\mathbb{N}}
\newcommand{\RR}{\mathbb{R}}

\newcommand{\pa}{\partial}

\newcommand{\al}{\alpha}

\newcommand{\de}{\delta}

\newcommand{\ep}{\varepsilon}

\newcommand{\origin}{\sf{o}}

\newcommand{\deficit}{{\rm{def}}}

\newtheorem{theorem}{Theorem}[section]
\newtheorem{corollary}[theorem]{Corollary}
\newtheorem{lemma}[theorem]{Lemma}
\newtheorem{proposition}[theorem]{Proposition}

\newtheorem{definition}[theorem]{Definition}
\newtheorem{remark}[theorem]{Remark}

\theoremstyle{remark}

\DeclareMathOperator{\osc}{\mathrm{osc}}
\DeclareMathOperator{\diam}{\mathrm{diam}}

\newcommand{\oscH}{\osc(H)}

\title[Almost CMC hypersurfaces in Space forms]{Quantitative Stability for hypersurfaces with almost constant curvature  in Space forms}

\author{Giulio Ciraolo, Alberto Roncoroni, Luigi Vezzoni}

\date{\today}
\thanks{The first two authors have been supported by GNAMPA of INdAM. The third author has been supported by GNSAGA of INdAM}
\address{G. Ciraolo, Dipartimento di Matematica e Informatica, Universit\`a di Palermo, Via Archirafi 34, 90123 Palermo, Italy.} \email{giulio.ciraolo@unipa.it}

\address{A. Roncoroni,  Dipartimento di Matematica F. Casorati, Universit\`a di Pavia, Via Ferrata 5, 27100 Pavia, Italy.} \email{alberto.roncoroni01@universitadipavia.it} 

\address{L. Vezzoni, Dipartimento di Matematica G. Peano, Universit\`a di Torino, Via Carlo Alberto 10, 10123 Torino, Italy.} \email{luigi.vezzoni@unito.it}


\begin{document}

\maketitle

\begin{abstract}
The Alexandrov Soap Bubble Theorem asserts that the distance spheres are the only embedded closed connected hypersurfaces in space forms having constant mean curvature. The theorem can be extended to more general functions of the principal curvatures $f(k_1,\ldots,k_{n-1})$ satisfying suitable conditions.

In this paper  we give sharp quantitative estimates of proximity to a single sphere for Alexandrov Soap Bubble Theorem in space forms when the curvature operator $f$ is close to a constant. Under an assumption that prevents bubbling, the proximity to a single sphere is quantified in terms of the oscillation of the curvature function $f$.
Our approach provides a unified picture of quantitative studies of the method of moving planes in space forms.
\end{abstract}

\tableofcontents

%
%
%
\section{Introduction}
In the celebrated papers \cite{Al1, Al2}, Alexandrov proved the so-called soap bubble theorem:

\medskip 

\noindent {\bf Alexandrov theorem}: {\em round spheres are the only $C^2$-regular, connected, closed hypersurfaces embedded in the Euclidean space having constant mean curvature.}

\medskip 
As Alexandrov observed in \cite{Al0,Al2}, the result can be extended to space forms $\mathbb{M}_+^n$, i.e. to the hyperbolic space $\mathbb H^{n}$ and the hemisphere $\mathbb S_{+}^n$ (but not in the whole sphere). 
Here and in the rest of the paper, we denote the Euclidean space $\RR^n$, the hyperbolic space $\mathbb H^n$ and the sphere $\mathbb S^n$ by
 the same symbol $\mathbb M^n$, while when we write $\mathbb M^n_+$ we mean that we are considering $\RR^n$, $\mathbb H^n$ and the hemispere 
 $\mathbb S^n_+$.

Alexandrov's theorem has been widely studied and extended in several directions. It is well-known that the theorem is in general false for non-embedded submanifolds (see e.g. \cite{HTY} and \cite{wente} for classical counterexamples in higher dimension and in $\mathbb{R}^3$, respectively). However, for immersed hypersurfaces, one can add some condition in order to guarantee that $S$ is a sphere: in particular Hopf proved in \cite{Hopf}  that every constant mean curvature $C^2$-regular sphere {\em immersed} in the $3$-dimensional Euclidean space is necessary a round sphere (see also \cite{AR,MMPR1,MMPR2} for a very recent generalization of Hopf's theorem to simply-connected homogeneous $3$-manifolds), and Barbosa and DoCarmo \cite{BarbosaDoCarmo} proved that every compact, orientable and  stable hypersurface \emph{immersed} in $\mathbb{R}^n$ is a round sphere (see also \cite{BdCE} for generalizations of this result). Moreover there exists \emph{non-closed} constant mean curvature hypersufaces embedded in $\mathbb{R}^3$ which are not diffeomorphic to a sphere, like for instance the unduloids (see \cite{Delunay} and \cite{Delunay_bis} for the generalization to higher dimensions).

In \cite{korevar} and \cite{rosenberg}  it is showed that Alexandrov's theorem  generalizes to a large class of curvature operators (see also \cite{Al2,CY,hartman,Hs,Lieb,Ro2,Ro1,Suss,Yau}). More precisely, let $S$  be a $C^2$-regular, connected, closed hypersurface embedded in $\mathbb M^n_{+}$. 
Then $S$ is always the boundary of a relatively compact connected open set $\Omega \subset \mathbb M^n_+$ and we orient $S$ by using the normal vector field to $S$ inward with respect to $\Omega$. Let $\{\kappa_1,\dots \kappa_{n-1}\}$ be the principal curvatures of $S$ ordered increasingly.  
We denote by ${\sf H}_S$ one of the following functions: 
\begin{enumerate}
\item[i)] the mean curvature $H:=\tfrac{1}{n-1}\sum_{i}\kappa_i$;

\item[ii)] $f(\kappa_1,\dots,\kappa_{n-1})$, where 
$$
f\colon \{x=(x_1,\dots,x_{n-1})\in \mathbb \RR^{n-1}\,\,:\,\,x_1\leq x_2\leq\dots\leq x_{n-1} \}\to \RR\,,
$$ 
is a $C^2$-function such that 
$$
f(x)>0, \mbox{ if } x_i>0 \mbox{ for every }i=1,\dots,n-1
$$
and $f$ is concave on the component $\Gamma$ of $\{x\in\RR^{n-1}\,\,:\,\,f(x)>0\}$ containing  $\{x\in \RR^{n-1}\,\,:\,\, x_i>0\}$. 
\end{enumerate}   
For instance, if $H_r$ denotes the $r$-higher order curvature of $S$ defined as the elementary symmetric polynomial of degree $r$ in the principal curvatures of $S$, then $H_r^{1/r}$ satisfies ii).  
By using this notation, Alexandrov's theorem can be stated as follows:

\medskip 

{\bf Alexandrov theorem II}: {\em distance spheres are the only $C^2$-regular, connected, closed hypersurfaces embedded in $\mathbb M^n_{+}$ such that ${\sf H}_S$ is constant.}

\medskip 

The main result in this paper is the following theorem which  gives sharp stability estimates of proximity to a single sphere for Alexandrov theorem II.

\begin{theorem}\label{main}
Let $S$ be a $C^2$-regular, connected, closed hypersurface embedded in $\mathbb M^n_{+}$ satisfying a uniform touching ball condition of radius $\rho$. There exist constants $\ep,\, C>0$ such that if 
\begin{equation}\label{H quasi const}
{\osc}({\sf H}_S)
 \leq \ep,
\end{equation}
then there are two concentric balls $ B^d_{r}$ and $B^d_{R}$ of $\mathbb M^n_+$ such that
\begin{equation}\label{Bri Om Bre}
S \subset \overline{B}^{\,d}_{R} \setminus B^d_{r},
\end{equation}
and
\begin{equation}\label{stability radii}
R-r \leq C {\osc}({\sf H}_S).
\end{equation}
The constants $\ep$ and $C$ depend only on $n$ and upper bounds on $\rho^{-1}$ and on the area of $S$.
\end{theorem}

The {\em uniform touching ball condition} of radius $\rho$ in theorem \ref{main} means that at any point of $S$ there exist two 
balls of radius $\rho$ both tangent to $S$ one from inside and one from outside. Since the constant $\rho$ is fixed, a bubbling phenomenon can not appear (see for instance \cite{CirMag}). As can be shown by a calculation for ellipsoids, the estimate in \eqref{stability radii} is optimal and it is new in the general setting of theorem \ref{main}. In the Euclidean space, quantitative studies for the mean curvature were avalaible in literature only for convex domains and they were not optimal (see the discussion in \cite{JEMS}); hence \eqref{stability radii} remarkably improves these results.

Theorem \ref{main} follows the research line initiated in \cite{JEMS} and pursued  in \cite{INDIANA}. The aim of the present paper is twofold: on  
the one hand we generalize the results in \cite{JEMS} and \cite{INDIANA} to the hemisphere and on the other hand we extend our analysis to 
a larger class of  curvature operators (in \cite{JEMS} and \cite{INDIANA} we only focused on the mean curvature). To achieve this goal, we provide a uniform quantitative approach of the method of the moving planes in space forms. Together with the wide generality of the assumptions of this theorem, the unified treatment of the problem in space forms is a major point of this manuscript.

As a consequence of theorem \ref{main} we have the following
\begin{corollary}\label{main2}
Let $\rho_0$, $A_0 > 0$ and $n\in\mathbb{N}$ be fixed. There exists $\varepsilon > 0$, depending on $n$, $\rho_0$ and $A_0$, such that if $S$ is a connected closed $C^2$ hypersurface embedded in $\mathbb{M}^n_{+}$ having area bounded by $A_0$, satisfying a touching ball condition of radius $\rho\geq\rho_0$, and such that
\begin{equation}
{\osc}({\sf H}_S)
 \leq \ep\,,
\end{equation}
then $S$ is $C^{1}$-close to a sphere and there exists a $C^2$-regular map $\Psi: \partial B_r^d\to \mathbb{R}$ such that
$$
F(p) = \exp_x(\Psi(p)N_p) 
$$
defines a $C^2$-diffeomorphism from $B^d_r$ to $S$ and
\begin{equation}\label{fristLipbound}
\|\Psi\|_{C^1(\partial B_r)} \leq C\, {\osc}({\sf H}_S)^{1/2}\,,
\end{equation}
where $N$ is a normal vector field to $\partial B^d_r$. 
\end{corollary}

As already mentioned, we tackle the problem by doing a quantitative study of the method of the moving planes, i.e. we study the original proof of Alexandrov from a quantitative point of view. There are other possible approaches for proving the symmetry result which are based on integral and geometric identities. The interested reader can refer to \cite{MR,Re,Ro1,Ro2}, and to \cite{Br} for a recent generalization (see also \cite{DelgadinoMaggi} for minimal assumptions on the regularity of $S$). The approach in \cite{Re} has been exploited in \cite{CirMag} to tackle the problem of bubbling (see also \cite{DMMN} and \cite{KrummelMaggi}). Symmetry and quantitative studies of proximity to a single sphere by using an integral approach can be found in \cite{ BianchiniCiraoloSalani,CFSW,CFMN,Feldmann,16BCS,17BCS,magnanini_survey,magnanini_poggesi1,magnanini_poggesi2,Qui Xia}.

\medskip

The paper is organized as follows. In section \ref{moving} we review the method of the moving planes in space forms and set up some notation. In section \ref{section3} we give technical local quantitative estimates in space forms. In section \ref{subsect Luigi} we find estimates on curvatures of projected surfaces in conformally Euclidean spaces. In section \ref{sect5} we prove the approximate symmetry in one direction. In section \ref{section_6} we show how to prove global approximate symmetry result by using the approximate symmetry in any direction. In section \ref{section_proof_main2} we complete the proof of main theorems.

We will use the following models of space forms: 
\begin{itemize}
\item $\mathbb H^{n}$ is the half-space $\{x\in \RR^{n}\,\,:\,\,x_n>0\}$ with the Riemannian metric 
\begin{equation}\label{hyp metric}
g_x=\tfrac{1}{x_n}\,\langle\cdot,\cdot\rangle\,,\,\, \mbox{ for every }x\in \RR^n
\end{equation} 
where$\langle\cdot,\cdot\rangle$ is the Euclidean product on $\RR^n$; 

\vspace{0.2cm}
\item $\mathbb S^n$ is the $n$-dimensional  unitary sphere  $\{x\in \RR^{n+1}\,\,:\,\, |x|=1\}$ with the round metric $g$ induced by the Euclidean metric in $\RR^{n+1}$. Here we recall that if we
consider the stereographic projection $\mathbb S^n\backslash\{\mbox{one point}\}\to \RR^n$, then $g$ is projected to 
the Riemann metric 
\begin{equation}\label{round metric}
g_x=\frac{4}{(1+|x|^2)^2}\,\langle \cdot,\cdot\rangle\,,\,\, \mbox{ for every }x\in \RR^n	\,. 
\end{equation} 
\end{itemize} 
In order to simplify the reading of the paper, there is a list of symbols at the end of the manuscript.
\bigskip

\noindent
{\bf Acknowledgements} The authors wish to thank Harold Rosenberg for
suggesting the problem studied in this paper. The paper was completed while the second author was visiting the Department of Mathematics of the ETH in Z\"urich, and he wishes to thank the institute for hospitality and support.

\section{The method of the moving planes}\label{moving}
In this preliminary section we recall the method of the moving planes in $\mathbb M^{n}_+$.

\medskip  
We begin by recalling the definition of {\em center of mass} in the context of Riemannian geometry (see e.g. \cite{Ka} for more details). 

Let $(M,g)$ be an oriented complete Riemannian manifold and let $\Omega$ be a bounded domain (i.e. a bounded connected open set). Let $P_\Omega\colon M\to \RR$  be the function   
$$
P_\Omega(x)=\frac{1}{2\,|\Omega|_g} \int_\Omega d(x,a)^2\,da \,,
$$
where $|\Omega|_g$ is the volume of $\Omega$ with respect to $g$.  Then the gradient of $P_\Omega$ takes the following expression 
\begin{equation}\label{gradP}
\nabla P_\Omega(x)=-\frac{1}{|\Omega|_g}\int_{\Omega} \exp^{-1}_x(a) \,da\,. 
\end{equation}
In some cases $P_\Omega$ is a convex map and it attains the minimum at only one point $\mathcal O$, which is usually called the {\em center of mass} of $\Omega$. For instance this occurs in the following cases: 
\begin{itemize}
\item all the sectional curvatures of $M$ are nonpositive; 

\vspace{0.1cm}
\item $\Omega$ is contained in a geodesic ball of radius $r<\tfrac12 \min\left\{{\rm inj}\, M,\tfrac{\pi}{2\sqrt{K}}\right\}$, where $K$ is an upper bound on the sectional curvatures of $M$;

\vspace{0.1cm}
\item $M=\mathbb S^n$ and $\Omega$ is contained in $\mathbb S^n_+$. 
\end{itemize}

We further recall that a Riemannian manifold $(M,g)$ is a symmetric space  if for every $p\in M$ there exists an isometry $f\colon  M\to M$ such that 
$f(p)=p$ and $f_{*|p}=-{\rm Id}$. 
\begin{lemma}\label{symspaces} Let $(M,g)$ be a symmetric space, $\Omega$ a bounded domain in $M$ and $x\in M$ be such that $\nabla P_{\Omega}(x)=0$. Assume that for every hyperplane $\pi$ in $M$ not containing $x$ there exists a hyperplane $\pi_1$ passing through $x$ and such that  $\pi\cap \pi_1\cap \Omega=\emptyset$. Then every hyperplane of symmetry for $\Omega$ contains $x$. 
\end{lemma}

\begin{proof}
Assume by contradiction that there exists a hyperplane $\pi$ of symmetry for $\Omega$ not containing $x$. Let $\pi_1$ be a hyperplane passing through $x$ and disjoint from $\pi$ inside $\Omega$, i.e. $\pi\cap \pi_1\cap \Omega=\emptyset$. Since $\pi_1$ and $\pi$ are disjoint, they subdivide $\Omega$ in three disjoint subsets $\Omega_1$, $\Omega_2$ and $\Omega_3$, with $|\Omega_i|_g>0,\, i=1,2,3$. Since $\Omega$ is symmetric about $\pi$, we have that 
$$
|\Omega_1|_g+|\Omega_2|_g=|\Omega_3|_g.
$$
Moreover formula \eqref{gradP} implies 
\begin{equation}\label{integrale}
\int_{\Omega_1} \exp^{-1}_{x}(a) \,da=-\int_{\Omega_2\cup\Omega_3} \exp^{-1}_{x}(a) \,da\,.
\end{equation}
Let $f\colon M\to M$ be the isometry such that $f(x)=x$ and $f_{|*x}=-{\rm Id}$. Then 
$$
\exp_x^{-1}f(a)=-\exp_x^{-1}(a)
$$
and 
$$
-\int_{\Omega_2\cup\Omega_3} \exp^{-1}_{x}(a) \,da=\int_{f(\Omega_2\cup\Omega_3)} \exp^{-1}_{x}(a) \,da\,.
$$
Therefore \eqref{integrale} implies 
$$
|\Omega_1|_g=|\Omega_2|_g+|\Omega_3|_g.
$$
which gives a contradiction since $|\Omega_2|_g>0$.
\end{proof}

Lemma \ref{symspaces} can be in particular applied in space forms $\mathbb M^n_+$, where we have the uniqueness of the center of mass.

\begin{proposition}[Characterization of the distance balls in  $\mathbb M^n_{+}$]\label{aboutO}
Let $S = \partial\Omega$ be a $C^2$-regular, connected, closed hypersurface embedded in $\mathbb{M}^{n}_+$, where $\Omega$ is a relatively compact domain. Assume that for every geodesic path $\gamma\colon \mathbb R\to \mathbb M^n$ there exists a hyperplane $\pi$ orthogonal to $\gamma$ such that  $S$ is symmetric about $\pi$.
Then $S$ is a distance  sphere about $\mathcal O$. 
\end{proposition}

\begin{proof}
In view of proposition \ref{symspaces} any hyperplane of symmetry of $S$ contains the point $\mathcal O$. Therefore $S$  is invariant by reflections about every hyperplane passing through $\mathcal O$. Since every orientation preserving  isometry 
of $\mathbb M^n$ fixing $\mathcal O$ can be obtained as the composition of  reflections about hyperplanes containing $\mathcal O$, $S$ is invariant by rotations and then it is a distance sphere.         
%
%
%
\end{proof}

We describe the method of the moving planes in  $\mathbb M^n_+$ and we introduce some notation. The method consists in moving hyperplanes along a geodesic path orthogonal to a fixed direction and it is similar in $\RR^n$, $\mathbb H^n$ and $\mathbb S^n_+$. The method can be described in terms of a point $\origin$ that we fix. Since $\mathbb M^n_+$ is a homogeneous space, 
 the construction does not depend on the choice of the point we fix. In particular we choose $\origin$ to be: the origin in $\RR^n$, $e_n$
  in $\mathbb H^n$ and the north pole in $\mathbb S^n_+$. For every direction $v\in T_{\origin}\mathbb M^n_+$, 
  we consider the geodesic path $\gamma_v\colon I\to \mathbb M^n_+$  satisfying $\gamma_v(0)=\origin$ and 
  $\dot \gamma_v(0)=v$. The domain of $\gamma_v$ is $I=\RR$ in $\RR^n$ and $\mathbb H^n$ and 
  $I=(-\tfrac {\pi}{2},\tfrac {\pi}{2})$ in $\mathbb S^{n}_+$.  For any $s\in I$ we denote by $\pi_{v, s}$ 
  the hyperplane passing through $\gamma_v(s)$ and orthogonal to $\dot{\gamma}_v(s)$, and we define  
$$
S_{v,s}=\{p\in S\,\,:\,\, p\in \pi_{v,t} \mbox{ for some }t>s\}\,. 
$$
We denote by $S^{\pi}_{v,s}$ be the reflection of $S_{v,s}$ about $\pi_{v, s}$.  
Note that 
\begin{itemize}
\item $S_{v,s}=\{p\in S\,\,:\,\, p\cdot v >s\}$, if $\mathbb{M}^n_+=\RR^n$;

\vspace{0.1cm}
\item $S_{v,s}=\{p\in  S\,\,:\,\, p\cdot \dot \gamma_v(s)>0\}$, if $\mathbb{M}^n_+=\mathbb S^n_+$. 
\end{itemize} 
In the hyperbolic case, giving an explicit description of $S_{v,s}$ is more complicated, but it can be simplified by assuming $v=e_1$. This assumption is not restrictive since we can always rotate every direction $v$ in $e_1$ by using an isometry of $\mathbb H^n$. In this case we have 
\begin{itemize}
\item $S_{e_1,s}=\{p\in S\,\,:\,\, p\cdot e_1 >s\}$, if $\mathbb{M}^n_+=\mathbb H^n$.
\end{itemize}   
Let
\begin{equation*}
m_v=\inf\lbrace s\in I \, : \, S^{\pi}_{v,t}\subset\Omega\mbox{ for every $t>s$}\rbrace\,. 
\end{equation*}
The hyperplane $\pi_{v}:=\pi_{v,m_v}$ is called the {\em critical hyperplane}. By construction $S_{v,m_v}^{\pi}$ is contained in $\overline \Omega$ and $S$ and $S_{v,m_v}^{\pi}$ are tangent at some point $p_0$ which can be either interior to $S_{v,m_v}^{\pi}$, or on $\partial S_{v,m_v}^{\pi}$ 
(and in this last case $p_0\in \pi_{v}$).  

\medskip 
Now for the sake of completeness we recall the generalized version of Alexandrov's theorem that we study in this paper and its proof in $\mathbb M_+^{n}$ (see \cite{Al2,korevar,MR,Re,Ro2,Ro1}).    
\begin{theorem} The only closed $C^2$-regular connected hypersurfaces embedded in $\mathbb M^n_+$ and such that  ${\sf H}_S$ is constant are the distance spheres.
\end{theorem}

\begin{proof}
The proof consists in showing that for every unitary vector $v\in T_{\origin}\mathbb M^n_+$, $S$ is symmetric about $\pi_{v}$. This is obtained by showing that  $S\cap S^{\pi}_{v,m_{v}}$ is open and closed in $ S^{\pi}_{v,m_{v}}$. Note that $S\cap S^{\pi}_{v,m_{v}}$ is  not empty since $S^{\pi}_{v,m_{v}}$ is tangent to $S$ at some point and $S\cap S^{\pi}_{v,m_{v}}$ is closed in $ S^{\pi}_{v,m_{v}}$. The only nontrivial step is that $S\cap S^{\pi}_{v,m_{v}}$ is open, which is obtained by using maximum principles for solutions to elliptic equations.  

Let $p_0\in S\cap S^{\pi}_{v,m_{v}}$. By construction we have that
$$
T_{p_0}S=T_{p_0}S^{\pi} \,,
$$
where $S^\pi$ is the reflection of $S$ about $\pi_v$. From the implicit function theorem, $S$ and $S^\pi$ are locally the Euclidean graph of $C^2$-regular functions $u$ and $\hat u$, respectively, defined in a ball $B_r$ of radius $r$ centered at the origin $O$ in $T_{p_0}S$. The functions $u$ and $\hat u$ satisfy the elliptic equation ${L} u (x) = {\sf H}_S(x,u(x))$ for $x \in B_r$; here the operator ${L}$ is the mean curvature operator or, more generally, a fully nonlinear operator. The ellipticity of ${L}$ is standard in the case of the mean curvature operator and it follows from \cite{korevar} for the other cases considered.  The proof is standard and hence it is omitted. 

Since ${\sf H}_S$ is constant, the difference $u-\hat u$ satisfies an elliptic equation of the form $\mathcal{L}(u-\hat u)=0$ with $u(O)-\hat u(O)= 0$.

If $p_0$ is interior to $S^{\pi}_{v,m_{v}}$ then we can choose $r$ sufficiently small such that $u-\hat u \geq 0$ in $B_r$ and by the strong maximum principle we obtain that $u-\hat u \equiv 0$ in $B_r$. 

If $p_0$ is on the boundary of $S^{\pi}_{v,m_{v}}$, then by construction $u-\hat u \geq 0$ in a half ball $B_r^+$, $u(O)-\hat u(O)= 0$ and $\nabla u(O)=\nabla u(O)=0$. By applying Hopf's boundary point lemma at the point $O$ we obtain that  $u-\hat u \equiv 0$ in $B_r^+$.

Hence, we have proved that $S\cap S^{\pi}_{v,m_{v}}$ is open, and the conclusion follows.
\end{proof}

\section{Local quantitive estimates in $\mathbb M^{n}$} \label{section3}
In this section we prove some preliminary estimates which we will use in the proof of the main theorem. We have the following preliminary lemma about the local equivalences of distances.  
\begin{lemma}\label{distance}
$\bullet$ Let $d$ be the distance induced by the hyperbolic metric \eqref{hyp metric} in $\mathbb H^n$ and let $q$ be such that $d(q,e_n)<R$; then 
\begin{equation} \label{oggilunedi}
c |q-e_n| \leq d (q,e_n) \leq C |q-e_n| \,,
\end{equation}
for some  positive constants $c$ and $C$ depending only on $R$.

\medskip 
$\bullet$ Let $d$ be the distance induced by the round metric \eqref{round metric} in $\RR^n$ and let $p,q$ in $\RR^n$ be such that $|p|,|q|\leq R\,.$ Then 
\begin{equation} \label{mid}
\frac{2}{1+R^2}|p-q| \leq d(p,q) \leq \pi |p-q| \,.
\end{equation}
\end{lemma}

Let us consider now a $C^2$-regular, connected, closed hypersurface $S=\partial \Omega$ embedded in $\mathbb M^n_+$, where 
$\Omega$ is a relatively compact domain and denote by $N$ the inward normal vector filed (inward with 
respect to $\Omega$). For $p\in S$ we denote by $\varphi_{p}\colon \mathbb 
 M^{n}_+\to \RR^n$  the following function whose definition depends on the geometry of $\mathbb M^n$: 
 \begin{itemize}
 \item if $\mathbb M^n$ is $\RR^n$, $\varphi_{p}\in {\rm SO}(n)\rtimes \RR^n$ and it is such that $\varphi_p(p)=0$ and $\varphi_{p*|p}\left(T_pS\right)=\{x_n=0\}$;

\vspace{0.1cm}
 \item if $\mathbb M^n$ is $\mathbb H^n$, $\varphi_p$ is an orientation preserving isometry of $\mathbb H^n$ such that  $\varphi_{p}(p)=e_n$, $\varphi_{p*|p}\left(T_pS\right)=\{x_{n}=0\}$;

\vspace{0.1cm}
\item if $\mathbb M^n$ is $\mathbb S^n$, $\varphi_{p}$ is the stereographic projection form the antipodal point to $p$ restricted to $\mathbb S^n_+$ composed with a rotation of $\mathbb R^n$ in order to have $\varphi_{p*|p}\left(T_pS\right)=\{x_{n}=0\}$.   
\end{itemize}
Note that in all the three cases we have that $\varphi_p(S)$ is a hypersurface embedded in $\RR^n$ and 
$$
\varphi_{p*|p}\left(T_pS\right)=\{x_{n}=0\}\,.
$$
For $r>0$, we denote by $\mathcal U_r(p)$ the open neighborhood of $p$ in $S$ such that $\varphi_p\left(\mathcal
U_r(p)\right)$ is the (Euclidean) graph of a $C^2$-function $u\colon B_r\to \RR$ defined in the ball of radius $r$ of $\RR^{n-1}$ 
centered at the origin. Even if we don't have a canonical choice of $\varphi_p$, the subsets $\mathcal{U}_r(p)$ do not depend on the choice of $\varphi_p$.  Moreover, the implicit function theorem implies that  any $p\in S$ has a neighborhood $\mathcal U_r(p)$ for $r$
sufficiently small. In order to establish the quantitive estimates we need in the proof of the main theorem, we have to show 
that $r$ can be uniformly bounded from below with a bound depending only on $\rho$. The following lemma also introduces the 
quantity $\rho_1$ which will be largely used in the paper.  

\begin{lemma}\label{Lemma1}
Let $S$ be a $C^2$-regular closed hypersurface embedded in $\mathbb M_+^n$ and satisfying  a touching ball condition of radius $\rho$ and let $\rho_1$ be defined in the following way:  
\begin{itemize}
\item $\rho_1=\rho$, if $\mathbb M^n=\RR^n$;

\vspace{0.1cm}
\item $\rho_1=(1-{\rm e}^{-\rho}\sinh\rho){\rm e}^{-\rho}\sinh\rho$,  if $\mathbb M^n=\mathbb H^n$;

\vspace{0.1cm}
\item $\rho_1=\tfrac{\rho}{\pi}$, if $\mathbb M^n=\mathbb S^n$\,.
\end{itemize}  
Then 
\begin{enumerate}
\item[{\rm(i)}] any point $p\in S$ admits a neighborhood $\mathcal U_{\rho_1}(p)$ and $\varphi_{p}(\mathcal U_{\rho_1}(p))$ is the graph of a $C^2$-function $u\colon B_{\rho_1}\to \RR$ satisfying 
\begin{equation}\label{5}
|u(x)-u(O)|\leq\rho_1-\sqrt{\rho_1^2-|x|^2}, \quad|\nabla u(x)|\leq\frac{|x|}{\sqrt{\rho_1^2-|x|^2}}\,;
\end{equation}

\vspace{0.1cm}
\item[{\rm(ii)}] there exists a universal constant $C$ such that for any $0<\alpha<\frac{1}{2}\min(1,\rho^{-1}_{1})$ and $q$ in $\mathcal U_{\alpha \rho_1}(p)$ we have 
\begin{equation}\label{d_S}
d_S(p,q)\leq\alpha C\rho_1\,,
\end{equation}
where $d_S$ is the geodesic distance on $S$.
\end{enumerate}
\end{lemma}

\begin{proof}
We refer to  \cite[Lemma 2.1]{JEMS} for the Euclidean case and we show how to deduce the statement form the 
Euclidean case (the hyperbolic case has been already done in \cite{INDIANA}). 

\begin{enumerate}
\item[(i)] It is enough to observe that $\varphi_p(S)$ satisfies an Euclidean touching ball condition of radius $\rho_1$ 
(this motivates the choice of $\rho_1$).  This can be easily deduced by using lemma \ref{distance}.  Hence the first item of the statement follows from the Euclidean case.

\vspace{0.1cm}
\item[(ii)]
Let $q\in \mathcal U_{\rho_1}(p)$. Then  $\varphi_p(q)=(x,v(x))$ for some $|x|<\rho_1$. Let $\gamma:[0,1]\rightarrow \varphi_p(S)$ be the curve joining $\varphi_p(p)$ to $\varphi_p(q)$ defined as $\gamma(t)=(tx,v(tx))$. Then
\begin{equation*}
\dot{\gamma}(t)=(x,\nabla v(tx)\cdot x)
\end{equation*}
and the Cauchy-Schwarz inequality implies
\begin{equation*}
|\dot{\gamma}(t)|\leq|x|\sqrt{1+|\nabla v(tx)|^2}\,.
\end{equation*}
Then \eqref{5} yields 
\begin{equation*}\label{gamma'}
|\dot{\gamma}(t)|\leq \dfrac{\rho_1|x|}{\sqrt{\rho_1^2-t^2|x|^2}}\leq\dfrac{|x|}{\sqrt{1-\alpha^2}}\leq\dfrac{2}{\sqrt{3}}|x|,
\end{equation*}
for $0\leq|x|\leq\alpha\rho_1$. Since
\begin{equation*}
d_S(p,q)\leq l(\gamma)
\end{equation*}
and 
\begin{eqnarray*}
&&l(\gamma)=\int_0^1\dfrac{|\dot \gamma|}{v(tx)}\, dt\,,\mbox{ in the hyperbolic case},\\
&&l(\gamma)=\int_0^1\dfrac{2}{1+|\gamma|^2}|\dot\gamma|\, dt\,,\mbox{ in the spherical case},
\end{eqnarray*}
we obtain that
\begin{equation*}
d_S(p,q)\leq C|x|
\end{equation*}
for a universal constant $C$ and \eqref{d_S} follows.
\end{enumerate}
\end{proof}

From lemma \ref{Lemma1} it follows the following 
\begin{corollary}\label{corollary_trascendental}
Let $S$ be a compact $C^2$-regular embedded hypersurfaces in $\mathbb M^{n}_+$ satisfying a touching ball condition of radius $\rho$. 
Let $q \in \mathcal{U}_{\alpha\rho_1}(p)$, with $0<\alpha<\tfrac{1}{2}\min(1,\rho^{-1}_{1})$. Then 
$$
d_S(p,q) \leq C d(p,q) \,,
$$
where $C$ depends only on $\rho$. In particular for every $p\in S$ the geodesic ball $\mathcal B_{r}(p)$ centered at $p$ and with radius $r<\tfrac{1}{2}\min(1,\rho^{-1}_{1})$ 
satisfies  
\begin{equation}\label{area}
{\rm Area}(\mathcal B_{r}(p))\geq cr^{n-1}\,,
\end{equation}
and $c$ is a constant depending only on $n$. 
\end{corollary}

\subsection{Quantitive stability of the parallel transport }
We study quantitative estimates involving the parallel transport. We adopt the following notation (see also \cite{INDIANA}):\\
given $p,q\in \mathbb M^n_+$ with $p\neq q$, we denote by $\tau_p^q\colon T_p\mathbb M^n\to T_q\mathbb M^n$ the parallel transport along the unique geodesic path $
\sigma$ connecting $p$ to $q$. We further assume that $\tau_p^p$ is the identity of $T_p\mathbb M^n$.

\begin{remark}
{\em If $\mathbb M^n$ is $\RR^n$, then $\tau_p^q$ is the identity for every $p,q$. \\
In the case $\mathbb M^n=\mathbb H^n$, if $p$ and $q$ belongs to the same vertical line, then we have $\tau_{p}^q(v)=\tfrac{p_n}{q_n} v\,.$\\
In the case $\mathbb S^n_+$, we have $\tau_p^q=(P_q\circ R_{\alpha})_{|T_p\mathbb S^n}$, where $R_\alpha$ acts as the rotation of the angle $\alpha=d(p,q)$ in
the plane $\pi$ containing $\sigma$ and as the identity in the orthogonal complement of $\pi$ and $P_q$ is the orthogonal projection onto $T_q\mathbb S^n$.}
\end{remark}

Here we introduce the following notation: 
\begin{itemize}
\item given $p\in \mathbb M^n$ and $v\in T_p\mathbb M^n$, $|v|_p:=g_{p}(v,v)^{1/2}$;

\vspace{0.2cm}
\item if $S=\partial \Omega$ is a compact $C^2$-regular embedded hypersurface in  $\mathbb M^{n}_+$, where $\Omega$ is a relatively compact domain in $\mathbb M^{n}_+$, 
$N$ is the inward unitary normal vector field on $S$.     
\end{itemize}

The first result we prove is the following 
%
\begin{proposition}\label{fejahyp2}
Let $S$ be a compact $C^2$-regular embedded hypersurface in  $\mathbb M^{n}_+$ satisfying a touching ball condition of radius $\rho$. There exists $\delta_0=\delta_0(\rho)$ such that if $p,q\in S$ with $d_S(p,q)\leq \delta_0$ then
\begin{equation} \label{bound on nu N+1}
g_p(N_p, \tau_q^p(N_q))\geq  \sqrt{1-C^2d_S(p,q)^2}\ \ \quad  \textmd{ and } \quad \ \ |N_p - \tau_q^p(N_q)|_p \leq C d_S(p,q)\,,
\end{equation}
where $C$ is a constant depending only on $\rho$.
\end{proposition}
\begin{proof}
We have already proved the assertion in the Euclidean and in the hyperbolic space in \cite{JEMS} and \cite{INDIANA}, respectively, and here
we focus on the spherical case. 
It is convenient to regard $S$ as a hypersurface of $\RR^n$ equipped with the spherical metric  \eqref{round metric}. We may further assume that $p$ is the origin $O$ of $\RR^n$ and $q$ belongs to a straight line passing through $O$.

Let $\delta_0=\min(\rho_1,\tfrac{1}{C})$, where $C$ will be specified later. If $d(p,q)\leq \rho_1=\tfrac{\rho}{\pi}$, then 
 we can apply the Euclidean estimates in \cite[Lemma 2.1]{JEMS} and obtain 
$$
\nu_{p} \cdot \nu_{q} \geq \sqrt{1-\frac{|p-q|^2}{\rho_1^2}} \,,
$$
where $\nu$ denotes the Euclidean inward normal vector field on $S$.  
Since 
$$
d(p,q)\leq \pi |p-q|\,,
$$
we have 
\begin{equation}\label{intermediate}
\nu_p \cdot \nu_q \geq \sqrt{1-C^2d_S(p,q)^2} \,,
\end{equation}
where $C$ depends only $\rho$. Moreover the inward $g$-unitary normal vector field $N$ to $S$ satisfies 
$$
N_p=\dfrac{1}{2}\nu_p\,,\quad \frac{2}{1+|q|^2}N_q=\nu_q=2\tau_q^p(N_q) \,,
$$
and \eqref{intermediate} implies
$$
g_p(N_p, \tau_q^p(N_q))\geq \dfrac{1}{2} \sqrt{1-C^2d_S(p,q)^2} \,,
$$
which is the first inequality in \eqref{bound on nu N+1}. The second inequality in \eqref{bound on nu N+1} follows by a direct computation 
and the claim follows. 
\end{proof}
 
We recall the following lemma proved in \cite{INDIANA}

\begin{lemma} \label{lem_p_pstar1}
Let $\Sigma$ and $\hat \Sigma$  be two compact embedded hypersurfaces in $\mathbb H^n$ satisfying both a touching ball condition of radius $\rho$. 
Assume that $e_n\in \Sigma$ and $T_{e_n}\Sigma=\{x_n=0\}$ and that there exist two local parametrizations $u,\hat u: B_{r} \to \RR$ of \,$\Sigma$ and
 $\hat \Sigma$, respectively, with $0<r\leq \rho_1$ and such that $u - \hat u  \geq 0$.
Let $p_1=(x_1,u(x_1))$ and $\hat p_1^*=(x_1,\hat u(x_1))$, with $x_1\in \partial B_{r/4}$, and denote by $\gamma$ the geodesic path starting from $p_1$ and tangent to 
$-\nu_{p_1}$ at $p_1$. Assume that 
\begin{equation} \label{conan}
d(p_1,\hat p_1^*) + |\nu_{p_1} - \nu_{\hat p_1^*} | \leq \theta \,. 
\end{equation}
for some $\theta \in [0, 1/2]$, where $\nu$ is the Euclidean unitary normal vector field to $\Sigma$. 
There exists $\bar r$ depending only on $\rho$ such that if $r \leq \bar r$ we have that $\gamma \cap \hat \Sigma \neq \emptyset$ and, if we denote by $\hat p_1$ the first intersection point between $\gamma$ and $\hat \Sigma$, then
\begin{equation*}
d(p_1,\hat p_1) + |N_{p_1} - \tau_{\hat p_1}^{p_1}(N_{\hat p_1})|_{p_1} \leq C \theta \,,
\end{equation*}
where $C$ is a constant depending only on $n$ and $\rho$, and provided that $C\theta < 1/2$.
\end{lemma}

\begin{figure}[h]
	\centering
	\includegraphics[scale=1]{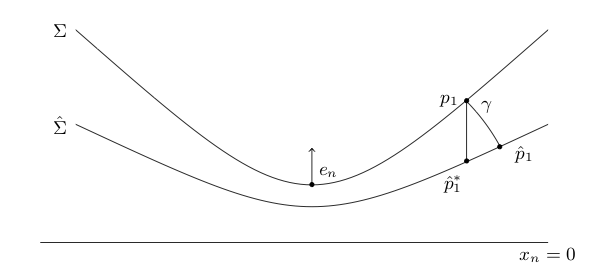}
	\caption{The statement of lemma \ref{lem_p_pstar1}}\label{figLuigi}
\end{figure}


The last result of this section is the \lq\lq spherical counterpart\rq\rq of lemma  \ref{lem_p_pstar1}. 

\begin{lemma} \label{lem_p_pstar2}
Let $\Sigma$ and $\hat \Sigma$  be two compact embedded hypersurfaces in $\mathbb R^n$ with the round metric \eqref{round metric} satisfying a touching ball condition of radius $\rho$. 
Assume $O\in \Sigma$ and $T_{O}\Sigma=\{x_n=0\}$ and that there exist two local parametrizations $u,\hat u: B_{r} \to \RR$ of \,$\Sigma$ and
 $\hat \Sigma$, respectively, with $0<r\leq \rho_1$ and such that $u - \hat u  \geq 0$.
Let $p_1=(x_1,u(x_1))$ and $\hat p_1^*=(x_1,\hat u(x_1))$, with $x_1\in \partial B_{r/4}$, and denote by $\gamma$ the geodesic path starting from $p_1$ and tangent to 
$-\nu_{p_1}$ at $p_1$. Assume that 
\begin{equation} \label{conan2}
d(p_1,\hat p_1^*) + |\nu_{p_1} - \nu_{\hat p_1^*} | \leq \theta \,. 
\end{equation}
for some $\theta \in [0, 1/2]$.  where $\nu$ is the Euclidean unitary normal vector field to $\Sigma$. There exists $\bar r$ depending only on $\rho$ such that if $r \leq \bar r$ we have that $\gamma \cap \hat \Sigma \neq \emptyset$ and, if we denote by $\hat p_1$ the first intersection point between $\gamma$ and $\hat \Sigma$, then
\begin{equation*}
d(p_1,\hat p_1) + |N_{p_1} - \tau_{\hat p_1}^{p_1}(N_{\hat p_1})|_{p_1} \leq C \theta \,,
\end{equation*}
where $C$ is a constant depending only on $n$ and $\rho$, and provided that $C\theta < 1/2$.
\end{lemma}

\begin{proof}
We first notice that, by choosing $r$ small enough in terms of $\rho$, from lemma \ref{Lemma1} we have that $|\nu_{p_1} - e_n | \leq  1/4$. We observe  
that the geodesic $\gamma$ is almost flat, i.e., viewed as an Euclidean circle its radius $R$ satisfies  
\begin{equation}\label{OR}
R= O\left(\dfrac{1}{|x_1|^2}\right) \mbox{ as $|x_1| \to 0$}\,. 
\end{equation}
Indeed, up to apply a rotation, we may assume that both $p_1$ and $\nu_{p_1}$ belong to the plane $\pi_1$ spanned by $\{e_1,e_2\}$.   In this way, the geodesic path $\gamma$   
belongs to the plane $\pi_1$ and we can work in a \lq\lq bidimensional way\rq\rq. We can write $p_1=(x_1,y_1)$ and $\nu_{p_1}=(\nu_1,\nu_2)$ and we compute the geodesic $\gamma$ passing through $p_1$ and tangent to $\nu_{p_1}$. We solve 
\begin{equation*}
\begin{cases}
(x_1+a)^2+(y_1+b)^2=1+a^2+b^2 \\ (x_1+a,y_1+b)\cdot(\nu_1,\nu_2)=0
\end{cases}
\end{equation*}
and we find 
\begin{equation*}
a=\dfrac{(1-|p_1|^2)\nu_2+2y_1  (p_1\cdot\nu)}{2(x_1\nu_2-y_1\nu_1)}  \,,\quad  b=\dfrac{-(1-|p_1|^2)\nu_1-2x_1(p_1\cdot\nu)}{2(x_1\nu_2-y_1\nu_1)}.
\end{equation*}
If $|x_1| \to 0$, according to lemma \ref{Lemma1}, we have that $y_1=O(x_1^2)$, $\nu_1=O(x_1)$ and $\nu_2=1+o(1)$; so we get that $a\sim\frac{1}{2x_1}$, $b$ is bounded and \eqref{OR} follows.

Let $B^+$ and $B^-$ be the exterior and interior touching balls of $\hat \Sigma$ at $
\hat p_0=(0,\hat u (0))$, respectively.  A standard geometrical argument shows that it is possible to choose $\bar r$ small enough in terms of $\rho$ such that $\gamma$ intersects $B^+$ and 
$B^-$ at points which are distant from the origin less than $\bar r$. This implies the existence of the point $\hat p_1$ in the assertion for any $r \leq \bar r$.  

Now we estimate the distance between $p_1$ and $\hat p_1$ as follows. Let $q$ be the unique point having distance $2\ep$ from $p_1$ and lying on the geodesic path containing $p_1$ and $\hat p_1^*$. Let $T$ be the geodesic right-angle triangle having vertices $p_1$ and $q$ and hypotenuse contained in the geodesic passing through $p_1$ and $\hat p_1$ (see figure \ref{figLuigi2}). Since the angle $\alpha$ at the vertex $p_1$ is such that $|\sin \al| \leq 1/4$, then from the cosine rule for spherical triangles we have that 
\begin{equation} \label{conad2}
d(p_1, \hat p_1) \leq C \theta \,.
\end{equation}
Moreover, the triangle inequality gives that
\begin{equation} \label{conad3}
d(\hat p_1^*, \hat p_1) \leq C \theta 
\end{equation}
for some constant $C$, and from \eqref{bound on nu N+1} we obtain that 
\begin{equation} \label{coop1}
|N_{p_1} - \tau_{\hat p_1}^{p_1}(N_{\hat p_1})|_{p_1} \leq |N_{p_1} - \tau_{\hat p_1^*}^{p_1}(N_{\hat p_1^*})|_{p_1}+ |\tau_{\hat p_1^*}^{p_1}(N_{\hat p_1^*}) - \tau_{\hat p_1}^{p_1}(N_{\hat p_1})|_{p_1} \,.
\end{equation}
Since $p_1$ and $\hat p_1^*$ are on the same vertical line \eqref{conan2} implies 
\begin{equation} \label{coop2}
|N_{p_1} - \tau_{\hat p_1^*}^{p_1}(N_{\hat p_1^*})|_{p_1} = |\nu_{p_1} - \nu_{\hat p_1^*}| \leq C \theta \,.
\end{equation}

As next step we show that 
\begin{equation}\label{coop3}
|\tau_{\hat p_1^*}^{p_1}(N_{\hat p_1^*}) - \tau_{\hat p_1}^{p_1}(N_{\hat p_1})|_{p_1}  \leq C \theta \,.
\end{equation}

We obtain \eqref{coop3} by showing that if $p,q\in \mathbb R^n$ belong to the Euclidean ball centered at the origin and having radius $s$, then 
\begin{equation}\label{coop4}
2|\tau_{p}^O(v)-\tau_{q}^O(w)| \leq \frac{(1+s)^2}{4}\,\left(d(p,O)^2+d(q,O)^2\right)+|v-\tau_{q}^p(w)|_p+\frac92 d(p,q)\,.
\end{equation}
for every $v,w\in \RR^n$, $|v|_q=|w|_{\hat q}=1$. We have 
$$
\tau_{p}^O(v)=\frac{1}{1+|p|^2} v \,,\quad \tau_{q}^O(w)=\frac{1}{1+|q|^2} w.
$$
and using Cauchy-Schwarz inequality and taking into account lemma \ref{distance} we have
$$
||q|^2-|p|^2|=|(q -p) \cdot (q + p)|\leq |q-p||q+p|\leq  s(1+s^2)\,d(p,q)
$$
since
$$
|w|=\frac{1+|\hat q|^2}{2}
$$
we have 
$$
2\left|\frac{1}{1+|q|^2}\tau_{p}^O(w)-\frac{1}{1+| q|^2}\tau_{ q}^O(w)\right|=2\left|\frac{| q|^2-|p|^2}{(1+|p|^2)(1+|q|^2)}\right|\leq 2s(1+s^2)\,d(p,q)\,.
$$
Now using  
$$
2\left|\tau_{p}^O(v)\right|=1 \,, \quad 2\left|\tau_{ q}^O(w)\right|=1;
$$
and 
$$
|v|=\frac{1+|p|^2}{2}, \quad |w|=\frac{1+|q|^2}{2}\,,
$$
we compute 
$$
\begin{aligned}
2|\tau_{p}^O(v)-\tau_{q}^O(w)| \leq &\,2 \left|\tau_{p}^O(v)-\frac{1}{1+|p|^2}\tau_p^O(v)\right|+2\left|\frac{1}{1+|p|^2}\tau_p^O(v)-\tau_{q}^O(w)\right|\\ 
\leq &\,2\left|1-\frac{1}{1+|p|^2}\right| \left|\tau_{p}^O(v)\right|+2\left|\frac{1}{1+|p|^2}\tau_q^O(v)-\frac{1}{1+|q|^2}\tau_{p}^O(w)\right|\\
&\,+2\left|\frac{1}{1+|q|^2}\tau_{p}^O(w)-\frac{1}{1+| q|^2}\tau_{ q}^O(w)\right|+2\left|1-\frac{1}{1+|q|^2}\right| |\tau_{q}^O(w)|\\ 
\leq &\,\left|1-\frac{1}{1+|p|^2}\right|+2\left|\frac{1}{1+|p|^2}\tau_q^O(v)-\frac{1}{1+|q|^2}\tau_{p}^O(w)\right|+4d(p, q)+\left|1-\frac{1}{1+|q|^2}\right| \\
\leq &\,|p|^2+2\left|\frac{1}{1+|p|^2}\tau_q^O(v)-\frac{1}{1+|q|^2}\tau_{p}^O(w)\right|+4d(p,q)+\left|q\right|^2 \\
\leq &\frac{(1+s)^2}{4}\,\left(d(p,O)^2+d(q,O)^2\right)+2\left|\frac{1}{1+|p|^2}\tau_p^O(v)-\frac{1}{1+|q|^2}\tau_{p}^O(w)\right|+4d(p, q)\,. 
\end{aligned}
$$

Now  we  show that  
$$
2\left|\frac{1}{1+|p|^2}\tau_p^O(v)-\frac{1}{1+|q|^2}\tau_{p}^O(w)\right|\leq |v-\tau_{q}^p(w)|_p+\frac12 d(p,q)\,. 
$$
Let  $\sigma$ be the geodesic path connecting $p$ with $q$. Then $\sigma$ is contained in a circle of $\RR^n $ and denotes by $C$ its center and by $\alpha$ the angle between $p-C$ and $q-C$.  Then 
$$
\frac{1+|p|^2}{1+|q|^2}w=R_{\alpha}\tau_{q}^pw\,,\mbox{ for every }w\in \RR^n\,,
$$ 
where $R_\alpha$ is the rotation (clockwise or anti-clockwise) about $\alpha$ in the plane containing $C$ and is the identity in the complement.  Therefore we have 
$$
\begin{aligned}
2\left|\frac{1}{1+|p|^2}\tau_p^O(v)-\frac{1}{1+|q|^2}\tau_{p}^O(w)\right|= &\,\left| \frac{1}{1+|p|^2}v-\frac{1}{1+|q|^2}w\right |_p\leq \left| v-\frac{1+|p|^2}{1+|q|^2}w\right |_p=
\left| v-R_{\alpha}\tau_{q}^pw\right |_p\\
\leq &\, \left| v-\tau_{q}^pw\right |_p+\left| \tau_{q}^pw-R_{\alpha}\tau_{q}^pw\right |_p
\end{aligned}
$$
and, consequently, we deduce,  
$$
\left| \tau_{q}^pw-R_{\alpha}\tau_{q}^pw\right |_p\leq |\alpha|\leq \frac{1}{2}d(p,q)
$$
which implies \eqref{coop4}. 

Therefore, by applying \eqref{coop4} to $|\tau_{\hat p_1^*}^{p_1}(N_{\hat p_1^*}) - \tau_{\hat p_1}^{p_1}(N_{\hat p_1})|_{p_1}$, we have 
$$
\begin{aligned}
|\tau_{\hat p_1^*}^{p_1}(N_{\hat p_1^*}) - \tau_{\hat p_1}^{p_1}(N_{\hat p_1})|_{p_1} & \leq \frac{(1+C\theta)^2}{4} \left(d(p_1,\hat p_1^*)^2 + d(p_1,\hat p_1)^2\right) +  |N_{\hat p_1^*} - \tau^{\hat p_1^*}_{\hat p_1} (N_{\hat p_1^*}) |_{\hat p_1^*}+\frac92 d(\hat p_1, \hat p_1^*)  \\
& \leq C \theta \,,
\end{aligned}
$$
where the last inequality follows from \eqref{conan2},\eqref{conad2},\eqref{conad3} and \eqref{bound on nu N+1}. This last inequality, \eqref{coop1} and \eqref{coop2} imply that
$$
|N_{p_1} - \tau_{\hat p_1}^{p_1}(N_{\hat p_1})|_{p_1} \leq C \theta \,,
$$
and therefore from \eqref{conad2} we conclude.
\end{proof}

\begin{figure}[h]
	\centering
	\includegraphics[scale=1]{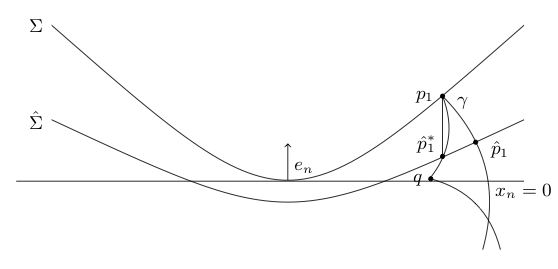}
		\caption{The statement and the proof of lemma \ref{lem_p_pstar2} }\label{figLuigi2}
\end{figure}

%
%
%
%
%
%
%
%
%
%
%
%

\section{Curvatures of projected surfaces in conformally Euclidean spaces} \label{subsect Luigi}

In this section we consider a connected open set $\Omega$ in $\RR^n$ equipped with a metric $g=h^2\,\langle\cdot,\cdot\rangle$ conformal to the Euclidean metric. We further assume the existence of an Euclidean hyperplane $\pi$ of $\RR^n$ such that $\Omega\cap \pi$ is a totally geodesic hypersurface in $\Omega$. This setting includes the Euclidean space, the hyperbolic space and $\RR^n$ with the round metric \eqref{round metric}. For instance in the half-space model of the hyperbolic space we can take as $\pi$ any vertical Euclidean hyperplane; in the spherical case we can consider Euclidean hyperplanes passing through the origin.

For our purposes, we consider a hypersurface $U$ of class $C^2$ embedded in $\Omega$ which intersects
$\pi$ transversally.  The implicit function theorem implies that $U'=U\cap \pi$ is a $C^2$-submanifold of $\pi$. Furthermore if $\nu_q$ is an Euclidean unit normal 
vector field to $U$ and $w$ is a unit normal vector to $\pi$, we have that 
$$
\nu'_q=(-1)^n*(*(\nu_q\wedge w)\wedge w)\,,\quad q\in U',
$$  
is an Euclidean  unitary normal vector field to $U'$ in $\pi$, where $*$ is the Euclidean Hodge star operator in $\RR^n$. In particular 
$U'$ is orientable in $\pi$. 
Let 
$$
N_q=\frac{1}{h(q)}\nu_q\,,\quad N'_q= \frac{1}{h(q)}\nu'_q
$$
be the normal vectors with respect to the metric $g$ and
$$
\omega_p=\frac{1}{h(p)}w\,,\quad p\in \Omega\,. 
$$

\begin{proposition}\label{prop Luigi I}
Let  $\kappa_j$, $j=1,\ldots,n-1$ be the principal curvatures of $U$ with respect to metric $g$ and to the orientation induced by $N$. Then the 
principal curvatures $\kappa_i'$ of $U'$ (viewed as submanifold of $\pi$) with respect to the orientation induced by $N'$  satisfy
\begin{equation} \label{bound curv I'}
\frac{1}{\sqrt{1-g_q(\omega_q,N_q)^2}}\kappa_1(q)\leq \kappa'_i(q)\leq \frac{1}{ \sqrt{1-g_q(\omega_q,N_q)^2}}\kappa_{n-1}(q) \,,
\end{equation}
for every $q\in U'$. Moreover, the principal curvatures ${\check \kappa}_i'$ of $U'$ seen  as a hypersurface of $U$ satisfy
\begin{equation} \label{milan}
|\check \kappa'_i(q)|\leq \frac{|g_q(\omega_q,N_q)|}{\sqrt{1-g_q(\omega_q,N_q)^2}}\,\max\{|\kappa_1(q)|,|\kappa_{n-1}(q)|\}\,,
\end{equation}
for every $q\in U'$. 

\end{proposition}

%
%

\begin{proof}
Let  $v\in T_qU'$  satisfy $|v|_q=1$ and
$$
\kappa_q(v)=g_q(\nabla_{v}\tilde N,v)\,,
$$
where $\tilde N$ denotes an extension of $N$ in $\Omega$ and $\nabla$ is the Levi-Civita connection of $g$.
For $p\in U'$, $N_q$ is orthogonal to $T_qU'$ and consequently it lies on the plane spanned by $ w$ and $N'_q$ and hence 
$$
N=a\,  w+ b N'\,,
$$
where $a$ is a function on $U'$ and 
$$
b=g(N,N')\,.
$$
If  $\tilde a$, $\tilde b$ and $\tilde N'$ are extensions of $a$, $b$ and $N'$ in $\Omega$, 
$$
\tilde N=\tilde a\,   w+ \tilde b\,\tilde N'
$$
defines an extension of $N$ and a direct computation yields 
$$
\kappa_p(v)=a(p)\,g_p(\nabla_{v}  w,v)+b(p)\,g_p(\nabla_{v} \tilde N',v)=b(p)\,g_p(\nabla_{v} \tilde N',v)\,,
$$
where we used that  $\pi\cap \Omega$ is totally geodesic. 
Therefore
$$
\frac{1}{g_q(N_q,N'_q)}\kappa_q(v)=g_q(\nabla_{v} \tilde N',v)
$$
and consequently
\begin{equation*} \label{bound curv I}
\frac{1}{g_q( N_q,N_q')}\kappa_1(q)\leq \kappa'_i(q)\leq \frac{1}{ g_q( N_q,N_q') }\kappa_{n-1}(q) 
\end{equation*}
for every $q\in U'$ and $i=1,\dots,n-2$. 

Now we show  
\begin{equation}\label{inter}
g_q( N_q,N_q')=1-g_q(\omega_q,N_q)^2 \,, 
\end{equation}
which implies  \eqref{bound curv I'}. We have
$$
(\nu_q\cdot \nu'_q)=(-1)^n*(*(\nu_q\wedge  w)\wedge  w)\cdot\nu_q=(-1)^n *(\nu_q\wedge  w)\wedge  w\cdot*\nu_q\,.
$$
Let $\{\nu_q,e_1,\dots,e_{n-1}\}$ be a positive-oriented orthonormal basis of $\mathbb R^{n}$
 such that 
\begin{itemize}
\item $\{e_1,\dots,e_{n-1}\}$ is a positive-oriented Euclidean-orthonormal basis of $T_{q}U$;

\vspace{0.1cm}
\item $\{e_2,\dots,e_{n-1}\}$ is a basis of $T_{q}U'$. 
\end{itemize}
In this way $ w\in \langle \nu_q,e_1\rangle$,
$$
*(\nu_q\wedge  w)=(w\cdot e_{1}) \,e_2\wedge \dots \wedge e_{n-1}\,,\quad  *\nu_q= e_1\wedge \dots \wedge e_{n-1},
$$
and
$$
\begin{aligned}
 \nu_q\cdot \nu'_q=&\,(-1)^n *(\nu_q\wedge w)\wedge w\cdot*\nu_q\\
 =&\,(-1)^n(w\cdot e_{1}) \,e_2\wedge \dots \wedge e_{n-1}\wedge w\cdot e_1\wedge \dots \wedge e_{n-1} \\
=&\,(-1)^n(w\cdot e_{1})^2 \,e_2\wedge \dots \wedge e_{n-1}\wedge e_1\cdot e_1\wedge \dots \wedge e_{n-1}\\
=&\,(w\cdot e_{1})^2 \,e_1\wedge \dots \wedge e_{n-1}\cdot e_1\wedge \dots \wedge e_{n-1}\\
=&\,(w\cdot e_{1})^2\,. 
\end{aligned}
$$
Since $|w|=1$, we have $(w\cdot e_{1})^2=1-(w\cdot \nu_q)^2$ and so 
\begin{equation}\label{*formula}
\nu_q\cdot \nu'_q=1-(w\cdot \nu_q)^2\,.
\end{equation}
Since
$$
\nu_q\cdot \nu'_q=g_q(N_q,N_q')\,,\quad \mbox{ and }\quad  w\cdot \nu_q=g_q(\omega_q,N_q) \,,
$$
\eqref{inter} follows.

Now we prove \eqref{milan}. In this case we regard $U'$ as a submanifold of $U$.  Let $q\in U'$, $v\in T_qU'$ such that $|v|_q=1$ and let $\alpha\colon (-\delta,\delta)\to S$ be a curve satisfying $\alpha(0)=q$, $\dot{\alpha}(0)=v$, $|\dot\alpha|_\alpha=1$. Let $\tilde N'$ be a unitary normal vector field of $U'$ in $U$ near $q$.   We may complete $v$ with an orthonormal basis $\{v,v_2,\dots v_{n-2}\}$ of $T_qU'$ such that 
$$
\check N'_q=*_q(N_q\wedge v\wedge v_2\wedge\dots \wedge v_{n-2})\,,
$$ 
where $*_q$ is the Hodge star operator at $q$ in $\Omega $ with respect to $g$ and to the standard orientation. Let 
$$
\check \kappa'_q(v)=g_{q}(*_q(\check N_q\wedge v\wedge v_2\wedge\dots \wedge v_{n-2}),D_t\dot \alpha_{|t=0})\,,
$$
where $D_t$ is the covariant derivative in $(\Omega,g)$.
Since $D_{t}\dot\alpha_{|t=0}\in \pi$, we have 
$$
\check \kappa'_q(v)=g_q(N_q,\omega_q)g_{q}(*_q(\omega_q\wedge v\wedge v_2\wedge\dots \wedge v_{n-2}),D_t\dot \alpha_{|t=0})\,. 
$$
Now, $*_q(\omega_q\wedge v\wedge v_2\wedge\dots \wedge v_{n-2})$ is a normal vector to $T_qU'$ in $\pi$ and so  
$$
\check \kappa'_q(v)=g_q(N_q,\omega_q)g_q(\nabla_{v}\tilde N,v)\,,
$$
where $\tilde N$ is an arbitrary extension of $N$ in a neighborhood of $q$. From \eqref{bound curv I'} we obtain
$$
|\check \kappa'_q(v)|\leq \frac{|g_q(N_q,\omega_q)|}{\sqrt{1-g_q(\omega_q,N_q)^2}}\,\max\{|\kappa_1(q)|,|\kappa_{n-1}(q)|\}\,,
$$
as required. 
\end{proof}

\begin{remark} 
It may be convenient to explain the meaning of \eqref{*formula} when $n=3$. In this case $*(v\wedge w)$ is the vector product $v\times w$ and so 
$$
\nu'_q=-(\nu_q\times w)\times w
$$
and 
$$
\nu_q\cdot \nu'_q=-(\nu_q\times w)\times w \cdot\nu_q =-(\nu_q\times w)\cdot(w\times\nu_q)=|\nu_q\times w|^2=|\nu_q|^2|w|^2-(\nu_q\cdot w)^2=1-(\nu_q\cdot w)^2\,. 
$$
\end{remark}

%

Now we focus in a different setting.  Let $\bar \Omega$ be the projection of $\Omega$ onto $\{x_n=0\}$ and let $\pi\subseteq \Omega$ be the graph of a $C^2$ function $F\colon A\to \RR $, where $A\subseteq \bar\Omega$ is a open subset. 

\begin{proposition} \label{prop Luigi II}
Let $U'$ be a $C^2$ regular oriented hypersurface of $\pi$ and let $U''$ be the orthogonal projection of $U'$ onto $\{x_n=0\}$. Then the principal curvatures of $U''$ satisfy 
\begin{equation}\label{bound curvatures II generale}
|\kappa_i''(\bar q)| \leq  \frac{h(q)}{\sqrt{1+|\nabla F(\bar q)|^2}}\left((\nu_q'\cdot e_{n})^2+\frac{1}{1+|\nabla F(\bar q)|^2}\right)^{-3/2}
\left(\max \{|\kappa_1'(q)|,|\kappa_{n-1}'(q)|\}+4\frac{\left|\nabla h(q)\right|}{h(q)^2}\right) \,,
\end{equation}
for every $i=1,\dots,n-2$,
where $\{\kappa''_i\}$ are the principal curvatures of $U''$ with respect to the Euclidean metric, $q\equiv(\bar q,q_{n})\in U'$ and $\nu'_q=h(q)N'_q$. 
\end{proposition}

\begin{proof}
If $X$ is a local positive oriented parametrization of $U'$, then $\bar X=X-(X\cdot e_{n})e_{n}$ is a local parametrization of $U''$, and we can orient $U''$ with
\begin{equation}\label{defnu2}
\nu''\circ \bar X:={\rm vers}(*(\bar X_1\wedge \bar X_2\wedge \dots \wedge \bar X_{n-2}\wedge e_{n}))\,,
\end{equation}
where $\bar X_k$ is the $k^{th}$ derivative of $\bar X$ with respect to the coordinates of its domain and $*$ is the Hodge star operator in $\RR^{n}$ with respect to the the Euclidean metric and the standard orientation. 

Now we prove inequalities \eqref{bound curvatures II generale}. Fix a point $q=(\bar q,q_{n})\in U'$ and $\bar v\in T_{\bar q}U'$ be nonzero.
Let $\beta \colon (-\delta,\delta)\to U''$ be an arbitrary regular curve contained in $U''$ such that
$$
\beta(0)=\bar q\,,\quad \dot\beta(0)=\bar v\,.
$$
Then
$$
\kappa''_{\bar q}(\bar v)=\frac{1}{|\bar v |^2}\nu''_{\bar q}\cdot \ddot \beta(0)
$$
is the normal curvature of $U''$ at $(\bar q,\bar v)$, viewed as hypersurface of $\{x_n=0\}$ with the Euclidean metric.
We can write
 $$
\kappa''_{\bar q}(\bar v)=\frac{1}{|\bar v |^2}\nu''_{\bar q}\cdot \ddot \alpha(0)
$$
where $\alpha=(\beta,\alpha_{n})$ whose projection onto $U'$ is $\beta$.
From
$$
\bar X_k=X_k-( X_k\cdot e_{n}) e_{n}\,,
$$
and the definition of $\nu''$ \eqref{defnu2} we have
$$
\kappa''_{\bar q}(\bar v)=\frac{(*(X_1( q)\wedge \dots \wedge X_{n-2}( q)\wedge e_{n}))\cdot \ddot\alpha(0)}{|\dot \beta(0)|^2 | X_1(q)\wedge \dots \wedge X_{n-2}(q)\wedge e_{n}|}\,.
$$
We may assume that $\{X_1(q),\dots,X_{n-2}(q)\}$ is an orthonormal basis of $T_{q}U'$ with respect to the Euclidean metric.  Let 
$$
{\sf N_q}=\frac{(-\nabla F(\bar q),1)}{\sqrt{1+|\nabla F(\bar q)|^2}}
$$
be the Euclidean normal vector to $\pi$ at $q$ and let 
$$
a=\sqrt{1+|\nabla F(\bar q)|^2}\,. 
$$

Therefore $\{X_1(q),\dots,X_{n-2}(q),\nu'_{q}, {\sf N}_q\}$ is an Euclidean orthonormal basis of $\RR^{n}$ and we can split
$\RR^{n}$ in
\begin{equation}\label{split}
\RR^{n}=T_{q}U'\oplus \langle \nu'_q\rangle\oplus \langle {\sf N}_q \rangle \,,
\end{equation}
and $e_{n}$ splits accordingly into
$$
e_{n}=e_{n}'+e_{n}''+e_{n}'''\,.
$$
Therefore
$$
*(X_1( q)\wedge \dots \wedge X_{n-2}( q)\wedge e_{n})\cdot\ddot \alpha(0)=*(X_1(q)\wedge \dots \wedge X_{n-2}( q)\wedge e_{n}''')\cdot \ddot \alpha(0)\,,
$$
i.e.
$$
*(X_1( q)\wedge \dots \wedge X_{n-2}( q)\wedge e_{n})\cdot\ddot \alpha(0)=\frac{1}{a}*\left(X_1( q)\wedge \dots \wedge X_{n-2}( q)\wedge{\sf N}_q\right)\cdot \ddot \alpha(0)\,.
$$
Since
$$
\nu_{q}'=*\left(X_1( q)\wedge \dots \wedge X_{n-2}(q)\wedge {\sf N}_q\right)
$$
we obtain
$$
\kappa''_{\bar q}(\bar v)= \frac{1 }{a |\dot \beta(0)|^2}\,\,\frac{\nu'_q\cdot\ddot\alpha(0)}{ |X_1( q)\wedge \dots \wedge X_{n-2}(q)\wedge e_{n}|}\,.
$$
We may assume that $\alpha$ is parametrized by arc length with respect to the metric $g$, i.e.
$$
|\dot \alpha|^2=h(\alpha)^{-2}
$$
and so
$$
|\dot \beta|^2=h(\alpha)^{-2}-\dot {\alpha} _{n}^2\,,
$$
which implies
\begin{equation} \label{braoAlberto}
\kappa''_{\bar q}(\bar v)= \frac{1}{a(h(q)^{-2}-v_{n}^2)}\,\,\frac{\nu'_q\cdot\ddot\alpha(0)}{ |X_1( q)\wedge \dots \wedge X_{n-2}(q)\wedge e_{n}|}\,.
\end{equation}
Since
$$
X_1(q)\wedge \dots \wedge X_{n-2}(q)\wedge e_{n}=X_1( q)\wedge \dots \wedge X_{n-2}(q)\wedge e''_{n}+X_1( q)\wedge \dots \wedge X_{n-2}(q)\wedge e_{n}'''
$$
and
\begin{eqnarray*}
&& X_1( q)\wedge \dots \wedge X_{n-2}(q)\wedge e''_{n}=(\nu_q'\cdot e_{n})\,X_1( q)\wedge \dots \wedge X_{n-2}(q)\wedge\nu_q'\,,\\
&& X_1( q)\wedge \dots \wedge X_{n-2}(q)\wedge e'''_{n}=\frac{1}{a}X_1(q)\wedge \dots \wedge X_{n-2}(q)\wedge {\sf N}_q\, \,,
\end{eqnarray*}
we obtain 
$$
|X_1(q)\wedge \dots \wedge X_{n-2}(q)\wedge e_{n}|=\left((\nu_q'\cdot e_{n})^2+\frac{1}{a^2}\right)^{1/2} \,.
$$
On the other hand
$$
\kappa_q'(v)=g_q(N'_q,D_t\dot \alpha_{|t=0})
$$
where $D_t$ is the covariant derivative in $\pi$. It is well-known that the Christoffel symbols of $g$ are given by
\begin{equation*}
\begin{aligned}
\Gamma_{ij}^k
& =\delta_i^k\partial_jf+\delta_j^k\partial_if-\delta_i^j\partial_kf \,,
\end{aligned}
\end{equation*}
where $f={\rm log}\,h$.
We have 
$$
\begin{aligned}
D_t\dot \alpha=&\,\ddot \alpha+\sum_{i,j,k=1}^{n} \Gamma_{ij}^k(\alpha)\dot \alpha_i\dot \alpha_j\,e_k=
\ddot \alpha+\sum_{i,j,k=1}^{n} (\delta_i^k\partial_jf(\alpha)+\delta_j^k\partial_if(\alpha)-\delta_i^j\partial_kf(\alpha))\dot \alpha_i\dot \alpha_j\,e_k\\
=&\, \ddot \alpha+\sum_{i,k=1}^{n}(2\partial_if(\alpha)\dot\alpha_i\dot \alpha_k-\dot\alpha_i^2\partial_kf(\alpha))\,e_k+\sum_{k=1}^n\partial_kf(\alpha)\dot\alpha_k^2\,e_k\\
\end{aligned}
$$
and 
$$
D_t\dot \alpha_{|t=0}=\ddot \alpha(0)+\sum_{i,k=1}^{n}(2\partial_if(q)v_iv_k-v_i^2\partial_kf(q))\,e_k+\sum_{k=1}^n\partial_kf(q) v_k^2\,e_k\,. 
$$
Therefore 
$$
\begin{aligned}
\kappa_q'(v)=&\,g_q\left(N'_q,\ddot \alpha(0)+\sum_{i,k=1}^{n}(2\partial_if(q)v_iv_k-v_i^2\partial_kf(q))\,e_k+\sum_{k=1}^n\partial_kf(q) v_k^2\,e_k\right)	\\
=&\,h(q)\nu_q'\cdot \ddot\alpha(0)+h(q)\sum_{i,k=1}^{n}(2\partial_if(q)v_iv_k-v_i^2\partial_kf(q))\,e_k\cdot \nu'_q+h(q)\sum_{k=1}^n\partial_kf(q) v_k^2 \nu_q'\cdot e_k\\
=&\,h(q)\nu_q'\cdot \ddot\alpha(0)+\sum_{i,k=1}^{n}(2\partial_ih(q)v_iv_k-v_i^2\partial_kh(q))\,e_k\cdot \nu'_q+\sum_{k=1}^n\partial_kh(q) v_k^2 \nu_q'\cdot e_k\, ,
\end{aligned}
$$
and we get  
$$
\nu_q'\cdot \ddot\alpha(0)=\frac{\kappa_q'(v)}{h(q)}-\frac{1}{h(q)}\sum_{i,k=1}^{n}(2\partial_ih(q)v_iv_k-v_i^2\partial_kh(q))\,e_k\cdot \nu'_q-\frac{1}{h(q)}\sum_{k=1}^n\partial_kh(q) v_k^2 \nu_q'\cdot e_k\,.
$$
From \eqref{braoAlberto} we deduce
\begin{multline*}
\kappa''_{\bar q}(\bar v)= \frac{1}{a(h(q)^{-2}-v_{n}^2)}\left((\nu_q'\cdot e_{n})^2+\frac{1}{a^2}\right)^{-1/2}\\
\left(\frac{\kappa_q'(v)}{h(q)}-\frac{1}{h(q)}\sum_{i,k=1}^{n}(2\partial_ih(q)v_iv_k-v_i^2\partial_kh(q))\,e_k\cdot \nu'_q-\frac{1}{h(q)}\sum_{k=1}^n\partial_kh(q) v_k^2 \nu_q'\cdot e_k\right)
\,,
\end{multline*}
for every $v\in T_qU'$, $g_q(v,v)=1$. Therefore
\begin{eqnarray*}
&& \kappa''_1(\bar q)=\frac{1}{ah(q)}\left((\nu_q'\cdot e_{n})^2+\frac{1}{a^2}\right)^{-1/2}
\inf_{v\in\mathbb S^{n-2}_q} A_q(v)\,,\\
&&\kappa''_{n-2}(\bar q)=\frac{1}{ah(q)}\left((\nu_q'\cdot e_{n})^2+\frac{1}{a^2}\right)^{-1/2}
\sup_{v\in\mathbb S^{n-2}_q}A_q(v)\,,
\end{eqnarray*}
where 
$$
A_{q}(v)=\frac{1}{h(q)^{-2}-v_{n}^2}\left(\kappa_q'(v)-\sum_{i,k=1}^{n}(2\partial_ih(q)v_iv_k-v_i^2\partial_kh(q))\,e_k\cdot \nu'_q-\sum_{k=1}^n\partial_kh(q) v_k^2 \nu_q'\cdot e_k\right)
$$
and $\mathbb{S}^{n-2}_q=\{v\in T_qU'\,\,:\,\, |v|_q=1\}$. Since $|v|_q^2=1$, then $|v|^2=h(q)^{-2}$ and we can rewrite $A_{q}(v)$ as 
$$
\begin{aligned}
A_{q}(v)=&\frac{1}{h(q)^{-2}-v_{n}^2}\left(\kappa_q'(v)-\nu_q'\cdot \left(2(\nabla h(q)\cdot v)v-h(q)^{-2}\nabla h(q)+\sum_{k=1}^n\partial_kh(q) v_k^2 e_k\right)\right)\\
=&\frac{1}{h(q)^{-2}-v_{n}^2}\left(\kappa_q'(v)-\nu_q'\cdot \left(2(\nabla h(q)\cdot v)v-h(q)^{-2}\nabla h(q)+\sum_{k=1}^n\partial_kh(q) v_k^2  e_k\right)\right) \,.
\end{aligned}
$$

Since $|\kappa_i''(\bar q)| \leq \max \{ |\kappa_1'' (\bar q)|, |\kappa_{n-2}'' (\bar q)| \}$, $i=1,\ldots,n-2$, we obtain
\begin{equation} \label{lasus}
|\kappa_i''(\bar q)| \leq\frac{1}{ah(q)}\left((\nu_q'\cdot e_{n})^2+\frac{1}{a^2}\right)^{-1/2}
\sup_{v\in\mathbb S^{n-2}_q}|A_q(v)|\,.
\end{equation}
We have
$$
\begin{aligned}
|A_{q}(v)|&\,=\frac{1}{h(q)^{-2}-v_{n}^2}\left|\kappa_q'(v)+\nu'_q\cdot\left(2(\nabla h(q)\cdot v)\, v-\frac{1}{h(q)^2}\nabla h(q)+\sum_{k=1}^n\partial_kh(q) v_k^2  e_k\right)\right|\\
&\,\leq \frac{1}{h(q)^{-2}-v_{n}^2}\left(|\kappa_q'(v)|+\left|2(\nabla h(q)\cdot v)\, v-\frac{1}{h(q)^2}\nabla h(q)+\sum_{k=1}^n\partial_kh(q) v_k^2  e_k\right| \right)\\
&\,\leq \frac{1}{h(q)^{-2}-v_{n}^2}\left(|\kappa_q'(v)|+2\left|(\nabla h(q)\cdot v)\, v\right|+\frac{1}{h(q)^2}\left|\nabla h(q)\right|+\left|\sum_{k=1}^n\partial_kh(q) v_k^2  e_k\right| \right)\\
&\,\leq \frac{1}{h(q)^{-2}-v_{n}^2}\left(|\kappa_q'(v)|+\frac{2}{h(q)^2}\left|\nabla h(q)\right|+\frac{1}{h(q)^2}\left|\nabla h(q)\right|+ \frac{1}{h(q)^2}\left|\nabla h(q)\right|\right)
\end{aligned}
$$
i.e.,
$$
|A_{q}(v)|\leq \frac{1}{h(q)^{-2}-v_{n}^2}\left(|\kappa_q'(v)|+\frac{4}{h(q)^2}\left|\nabla h(q)\right|\right) \,.
$$
Since $\RR^{n}=T_qU'   \oplus \langle {\sf N}_q \rangle \oplus  \langle \nu_q' \rangle $, we can write 
$$
e_n=e_n''+ \frac{1}{a}+\nu_q\cdot e_n\,\nu'_q\,,
$$
where $\tilde e_n$ is the orthogonal projection of $e_n$ onto $T_qU'$. Therefore 
$$
1-|e_n'|^2=\frac{1}{a^2}+(\nu'_q\cdot e_n)^2\,.
$$
Since $h(q)v$ lies in $T_qU'$ and it has unitary Euclidean norm, we have 
$$
|e_n'|^2\geq h(q)^2(e_n\cdot v)^2=h(q)^2v_n^2
$$  
and so 
$$
1-h(q)^2v_n^2\geq \frac{1}{a^2}+(\nu'_q\cdot e_n)^2\,,
$$
i.e.
$$
h(q)^{-2}-v_n^2\geq \left(\frac{1}{a^2}+(\nu'_q\cdot e_n)^2\right)h(q)^{-2}\,.
$$
Hence 
$$
|A_q(v)|\leq (|\kappa_q'(v)|+4h(q)^{-2}|\nabla h(q)|)\left(\frac{1}{a^2}+(\nu'_q\cdot e_n)^2\right)^{-1}h(q)^{2} \,,
$$ 
which yields 
\begin{equation} \label{lasus_II}
|\kappa_i''(\bar q)| \leq  \frac{h(q)}{a}\left((\nu_q'\cdot e_{n})^2+\frac{1}{a^2}\right)^{-3/2}
\sup_{v\in\mathbb S^{n-2}_q}(|\kappa_q'(v)|+4h(q)^{-2}|\nabla h(q)|) \,,
\end{equation}
which implies \eqref{bound curvatures II generale}.
\end{proof}

Now we use \eqref{bound curvatures II generale} in space forms.

\noindent In the {\em Euclidean space} we have $\Omega=\RR^n$ and  $h(q)=1$ and \eqref{bound curvatures II generale} reduces to 
$$
|\kappa_i''(\bar q)| \leq  \frac{1}{\sqrt{1+|\nabla F(\bar q)|^2}}\left((\nu_q'\cdot e_{n})^2+\frac{1}{1+|\nabla F(\bar q)|^2}\right)^{-3/2}
\max \{|\kappa_1'(q)|,|\kappa_{n-1}'(q)|\} \,,
$$
which was already found in  \cite[Proposition 2.8]{JEMS} when $\pi$ is a hyperplane, $\omega_1=\frac{(\nabla F(\bar q),-1)}{\sqrt{1+|F(\bar q)|^2}}$ and $\omega_2=e_n$.

\noindent In the {\em Hyperbolic space} we have $\Omega=\{q_n>0\}$ and $h(q)=\frac{1}{q_n}$ and \eqref{bound curvatures II generale} reduces to  
$$
|\kappa_i''(\bar q)| \leq  \frac{1}{q_n\sqrt{1+|\nabla F(\bar q)|^2}}\left((\nu_q'\cdot e_{n})^2+\frac{1}{1+|\nabla F(\bar q)|^2}\right)^{-3/2}
\left(\max \{|\kappa_1'(q)|,|\kappa_{n-1}'(q)|\}+4\right) \,,
$$
(see also \cite[Proposition 4.3]{INDIANA}).

Now we focus $\RR^n$ equipped with the spherical metric. In this case $h(q)=\frac{2}{1+|q|^2}$ and  \eqref{bound curvatures II generale} gives
$$
|\kappa_i''(\bar q)| \leq  \frac{2}{(1+|q|^2)\sqrt{1+|\nabla F(\bar q)|^2}}\left((\nu_q'\cdot e_{n})^2+\frac{1}{1+|\nabla F(\bar q)|^2}\right)^{-3/2}
\left(\max \{|\kappa_1'(q)|,|\kappa_{n-1}'(q)|\}+4|q|\right)\,.
$$
In particular if $\pi$ is the hemisphere of some hyperplane which does not contain the origin, then we have  
\begin{equation}\label{bound curvatures II}
|\kappa_i''(\bar q)| \leq  \frac{h(q) |(q-O_{\pi})_n|}{R}\left((\nu_q'\cdot e_{n})^2+\frac{(q-O_{\pi})_n^2}{R^2}\right)^{-3/2}
\left(\max\{|\kappa_1'(q)|,|\kappa_{n-1}'(q)|\}+4|q|\right),
\end{equation}
for every $i=1,\dots,n-2$,
where $O_\pi$ and $R$ are the center and the radius of $\pi$, respectively.

\section{Approximate symmetry in one direction} \label{sect5}
We consider the following set-up: let $S=\partial \Omega$ be a $C^2$-regular connected closed hypersurface embedded in $\mathbb M^n_+$, where $\Omega$ is a bounded domain.
 Assume that $S$ satisfies a uniform touching ball condition of radius $\rho>0$. 
We fix a direction $v$ in $T_{\origin}\mathbb M^n$ and  we apply the method of the moving planes  as described in section \ref{moving}. Let
$\pi=\pi_{v,m_{v}}$ be the critical hyperplane and in order to simplify the notation we set 
$$
\begin{aligned}
&S_+=\{p\in S\,\,:\,\, p\in \pi_{v,t} \mbox{ for some }t>m_v\}\,,\\
&S_-=\{p\in S\,\,:\,\, p\in \pi_{v,t} \mbox{ for some }t<m_v\}\,. 
\end{aligned}
$$
From the method of the moving planes we have that the reflection $S_+^\pi$ of $S_+$ with respect to $\pi$ is contained in $\Omega$ and it is tangent to $S_-$ at a point $p_0$ (internally or at the boundary). 
Let $\Sigma$ and $\hat \Sigma$ be the connected 
components of $S_+^\pi$ and $S_-$ containing $p_0$, respectively.  

The main result in this section is the following 

\begin{theorem}\label{thm approx symmetry 1 direction}
There exists $\ep>0$ such that if
$$
{\rm osc}({\sf H}_S) \leq \ep,
$$
then for any $p \in\Sigma$ there exists $\hat p\in \hat\Sigma$ such that
\begin{equation*}
d(p,\hat p) +  |N_p-\tau_{\hat p}^p (N_{\hat p})|_p \leq C\, {\rm osc}({\sf H}_S)  . 
\end{equation*}
Here, the constants $\ep$ and $C$ depend only on $n$, $\rho$ and the area of $S$. In particular $\ep$ and $C$ do not depend on the direction $v$. 

Moreover, $\Omega$ is contained in a neighborhood of radius $C{\rm osc}({\sf H}_S)$ of $\Sigma \cup \Sigma^\pi$ $($ $\Sigma^\pi$ is the reflection of $\Sigma$ about $\pi$ $)$, i.e.
$$
d(p,\Sigma \cup \Sigma^\pi) \leq C{\rm osc}({\sf H}_S)\,,
$$
for every $p \in \Omega$.
\end{theorem}

\medskip 
Before giving the proof of theorem \ref{thm approx symmetry 1 direction}, we provide two preliminary results
about the geometry of $\Sigma$. For $t> 0$ we set
$$
\Sigma_t = \{p \in \Sigma:\ d_\Sigma(p,\pa \Sigma) >  t \}\,.
$$
The following lemmas quantitatively show that $\Sigma_t$ is connected for $t$ small enough.

Here we use the results in section \eqref{subsect Luigi} and we consider the unitary normal vector field $\omega$ to $\pi$ directed as the geodesic $\gamma$ in 
$\mathbb M_+^n$ satisfying $\dot \gamma(0)=v$. 
\begin{lemma} \label{lemma connected}
Assume 
\begin{equation}\label{trasversale}
g_p(N_p, \omega_p)\leq \mu
\end{equation}
for every $p$ on the boundary of $\Sigma$, for some $\mu\leq 1/2$, and let $t_0=\rho\sqrt{1-\mu^2}$. 
Then $\Sigma_t$ is connected for any $0 < t < t_0$. 
\end{lemma}
\begin{proof} 
We can work in $\RR^n$ for every space form considered, and we may assume that $\pi$ is an Euclidean 
hyperplane of $\RR^n$ (in the spherical case we can consider the projection from a point antipodal to a point inside $\pi$).

Let $\Sigma'$ be the subset of $\pi$ obtained by projecting $\Sigma$ onto $\pi$ (for any point $p \in \Sigma$ we define the projection of $p$ onto $\pi$ as the point on $\pi$ which realizes the distance $d$ of $p$ from $\pi$). $\Sigma'$ is an open set  of $\pi$ with $\partial \Sigma'=\partial \Sigma$. 
Proposition \ref{prop Luigi I} gives 
$$
|\kappa'_i(p)|\leq \frac{1}{ \sqrt{1-\left(g_p(N_p, \omega_p)\right)^2}}\max \{|\kappa_1(p)|,|\kappa_{n-1}(p)|\} \,,
$$
for any $p\in \partial \Sigma$ and $i=1,\ldots,n-1$, where $\kappa'_i$ are the principal curvatures of $\partial \Sigma$ viewed as a hypersurface of $\pi$. 
Since $S$ satisfies a touching ball condition of radius $\rho$, we have 
$$
\max \{|\kappa_1(p)|,|\kappa_{n-1}(p)|\}\leq \frac{1}{\rho}
$$
and, consequently, 
\begin{equation} \label{pulizie}
|\kappa'_i(p)|\leq \frac{1}{ \rho \sqrt{1-\left(g_p(N_p, \omega_p)\right)^2}} \,,
\end{equation}
for $i=1,\ldots,n-1$. From \eqref{trasversale} and \eqref{pulizie} we have that $\partial \Sigma'$ satisfies a touching ball condition of radius 
$$
\rho'\geq \rho\sqrt{1-\left(g_p(N_p, \omega_p)\right)^2} \geq t_0\,.
$$ 
Therefore if $s<t_0$,
$$
\mathcal{C}_s = \{ z \in \pi :\ d(z, \partial \Sigma) < s \}
$$
is a collar neighborhood of $\partial \Sigma$ in $\Sigma'$ of radius $s$. Since $\pi$ is a critical hyperplane in the method of moving planes, if $p$ belongs to the maximal cap $
S_{+}$ then any point on 
the geodesic path connecting $p$ to its projection onto $\pi$ is contained in the closure of $\Omega$.
It follows that the preimage of $\mathcal{C}_s$ via the projection contains a collar neighborhood of $\partial \Sigma$ of radius $s$ in $\Sigma$. This implies that $\Sigma$ can 
be retracted in $\Sigma_t$ for any $t\leq s$ which completes the proof.
\end{proof}

\begin{lemma} \label{lemma connected II}
There exists $\bar \delta>0$ depending only on $\rho$ with the following property. 
Assume that there exists a connected component $\Gamma_\de$ of $\Sigma_\de$, for some $0<\de \leq \bar \delta$, such that one of the following two assumptions is fulfilled:  
\begin{enumerate}
\item[$i)$] $0 \leq g_q(N_q, \omega_q) \leq \frac{1}{8}$  for any $q \in \partial \Gamma_\de$, 

\vspace{0.1cm}
\item[$ii)$] for any $q\in \partial \Gamma_\delta$ there exists $\hat q \in \hat \Sigma$ such that 
$$
d(q,\hat q)+|N_q-\tau_{\hat q}^q(N_{\hat q})|_q \leq \delta\,.
$$
\end{enumerate}
Then 
\begin{equation}  \label{bellachegira}
0\leq g_q(N_q, \omega_q)\leq \frac14
\end{equation}
for any $q\in \partial \Sigma$ 
and $\Sigma_\delta$ is connected. 
\end{lemma}
	
\begin{proof}
\emph{Case $i)$.}
The crucial observation is that we can choose $\bar \delta$ small enough such that $\bar \delta \leq \delta_0$, where $\delta_0$ is the bound appearing in proposition \ref{fejahyp2}, and the set $ \Sigma \setminus \Gamma_\de$ is enclosed by $\pi$ and the set obtained as the union of all the exterior and interior touching balls to  the reflection of $S$ about $\pi$, $S^\pi$. This implies that for any $p \in \Sigma \setminus \Gamma_\de$ there exists $q \in \partial \Gamma_\de$ such that $d_{\Sigma}(p,q) \leq \de$ and we can apply the estimates in proposition \ref{fejahyp2}. Indeed from \eqref{bound on nu N+1} we have that
$$
|N_p - \tau_q^p(N_q)|_p\leq C \de\,,\mbox{ and } g_p(N_p, \tau_q^p(N_q))\geq \sqrt{1-C^2\delta^2}\,,
$$ 
where $C=C(\rho)$. Therefore 
$$
g_p(N_p,\omega_p)= g_p(N_p-\tau_q^p(N_q),\omega_p)+g_p(\tau_q^p(N_q)\cdot \omega_p) \leq C\delta+g_p(\tau_q^p(N_q),\omega_p)
$$
and by using 
$$
g_p(\tau_q^p(N_q),\omega_p)=g_q(N_q, \tau_p^q(\omega_q))
$$
we obtain
$$
g_p(N_p,\omega_p) \leq C\delta+ g_q(N_q,\omega_q) + g_q(N_q ,\tau_p^q(\omega_p)-\omega_q)\leq C\delta+g_q(N_q,\omega_q) +|\tau_p^q(\omega_p)-\omega_q|_q \,. 
$$
Since 
$$
|\tau_p^q(\omega_p)-\omega_q|_q=0\,,
$$
we deduce 
$$
g_p(N_p,\omega_p) \leq  C\delta+g_q(N_q,\omega_q)\,. 
$$
This last bound holds for every $p \in \partial \Sigma$ and by choosing $\delta$ small enough in terms of $\rho$ we obtain \eqref{bellachegira}, as required. 

\medskip 
\emph{Case $ii)$: $\Gamma_\delta$ satisfies ii).}	Let $q\in \partial \Gamma_\delta$. 
By construction of the method of moving planes, $g_q(N_q, \omega_q) \geq 0$. We denote by $q^\pi$ the reflection of $q$ about $\pi$ and we have
$$
d(q^\pi,\hat q) \leq d(q^\pi,q) + d(q,\hat q) \leq 3 \delta \,.
$$
Up to consider a smaller $\delta$ in terms of $\rho$, from corollary \ref{corollary_trascendental} we find $C=C(\rho)$ such that $d_S(q^\pi,\hat q) \leq C \delta $ and 
$q^\pi\in \mathcal{U}_{\rho_1}(\hat q)$. Hence we can apply \eqref{bound on nu N+1} and obtain 
$$
g_{\hat q}(N_{\hat q} \cdot \tau_{q^\pi}^{\hat q}(N_{q^\pi}) \geq  \sqrt{1-C^2\delta^2}\\ \quad \textmd{ and } \quad \ \ |N_{\hat q} - \tau_{q^\pi}^{\hat q}(N_{q^\pi})|_{\hat q} \leq C \delta \,.
$$
Since $N_{q^\pi}$ and $q^\pi$ are the reflection of $N_q$ and $q$ about $\pi$, respectively, we have that
$$
g_q(N_q,\omega_q) = - g(\tau_{q^\pi}^{q}(N_{q^\pi}),\omega_q) \,,$$
and hence
$$
2 g_q(N_q,\omega_q) =g_q(N_q-\tau_{q^\pi}^{q}(N_{q^\pi}),\omega_q) =g_q(N_q-\tau_{\hat q}^{ q}(N_{\hat q})),\omega_q) +
g_q(\tau_{\hat q}^{q}(N_{\hat q})-\tau_{q^\pi}^{q}(N_{q^\pi})),\omega_q) \,.
$$
This implies that 
$$
0\leq 2g_q( N_q ,\omega_q) \leq |N_q-\tau_{\hat q}^{q}(N_{\hat q})|_q+|\tau_{\hat q}^{q}(N_{\hat q})-\tau_{q^\pi}^{q}(N_{q^\pi})|_q \,. 
$$
Next we observe that 
$$
|\tau_{\hat q}^{q}(N_{\hat q})-\tau_{q^\pi}^{q}(N_{q^\pi})|_q=|N_{\hat q}-\tau^{\hat q}_{q}\tau_{q^\pi}^{q}(N_{q^\pi})|_{\hat q}\leq c(\delta) |N_{\hat q}-\tau^{\hat q}_{q^\pi}(N_{q^\pi})|_{\hat q}
$$
where $c(\delta)\to 0$ when $\delta\to 0$.  Hence for a suitable choice of $\bar 	\delta$ we get
\begin{equation} \label{treno}
0\leq 2g_q(N_q,\omega_q) \leq\frac18\,,
\end{equation}
and the claim follows from case  $i)$.	 
\end{proof}

Now we can focus on the proof of the first part of theorem \ref{thm approx symmetry 1 direction}, and show that there exist constants $\ep$ and $C$, 
depending only on $n$, $\rho$ and $|S|_g$, such that if
$$
{\rm osc}( {\sf H}_S) \leq \ep,
$$
then for any $p$ in $\Sigma$ there exists $\hat p$ in $\hat\Sigma$ satisfying
\begin{equation}\label{bound on dist}
d(p,\hat p) +  |N_p-\tau_{\hat p}^p (N_{\hat p})|_p \leq C\, {\rm osc}( {\sf H}_S)  \,. 
\end{equation}

In the proof of theorem \ref{thm approx symmetry 1 direction} we are going to choose a number $\delta>0$ sufficiently small in terms of $\rho$, $n$ and $|S|_g$.  
A first requirement on $\delta$ is that the assumptions of lemmas \ref{lemma connected} and \ref{lemma connected II} are satisfied. Other restrictions on the value of $\delta$ will be done in the development of the proof. We subdivide the proof of the first part of the statement in four cases depending on the whether the distances of $p_0$ and $p$ from $\partial \Sigma$ are greater or less than $\delta$.

\subsubsection{Case 1. $d_{\Sigma}(p_0,\pa \Sigma) > \de$ and $d_{\Sigma}(p,\pa \Sigma) \geq \de$} \label{subsec case1}

In this first case we assume that $p_0$ and $p$ are interior points of $\Sigma$, which are far from $\pa \Sigma$ more than $\de$.
We first assume that $p_0$ and $p$ are in the same connected component of $\Sigma_\de$; then, lemma \ref{lemma connected II} will be used in order to show that $\Sigma_\de$ is in fact connected.

Let $r_0>0$ be such that $\mathcal U_{r_0}(p_i)\subset \Sigma$ for every $p_i\in \Sigma_\delta$. The value of $r_0$ follows from lemma \ref{Lemma1} by letting
\begin{equation} \label{rpippo}
r_0=\min(\bar r,\alpha \rho_1 )\,,
\end{equation}
where $\bar r$ is given by lemmas \ref{lem_p_pstar1} and  \ref{lem_p_pstar2}, $\alpha \in (0, \frac{1}{2} \min(1,\rho_1^{-1}))$ is such that $\alpha C \rho_1 \leq \tfrac{\delta}{4}$, and $C$ is the constant appearing in \eqref{d_S}.

\begin{lemma}\label{catenadipalle}
	Let $\ep_0 \in [0,1/2]$, $p_0$ and $p$ be in a connected component of $\Sigma_\delta$ and $r_i=(1-\ep_0^2)^ir_0$.  There exist an
	integer $J\leq J_\de$, where 
	\begin{equation}\label{N_delta_hyp}
	J_\de:=\max\left(4,\frac{2^{n-1} |S|_g}{ \de^{n-1}} \right) \,,
	\end{equation}
	and a sequence of points $\{p_1,\dots,p_J\}$ in $\Sigma_{\delta/2}$ such that
	\begin{eqnarray*}
		&& p_0,p \in \bigcup_{i=0}^J \overline{\mathcal U}_{r_i/4}(p_i)\,,\\
		&& \mathcal U_{r_0}(p_i)\subseteq \Sigma, \quad i = 0,\ldots, J\,,\\
		&& p_{i+1} \in\overline{\mathcal U}_{r_i/4}(p_i), \quad i=0,\ldots,J-1\,.
	\end{eqnarray*}
\end{lemma}
\begin{proof}
In view of corollary \ref{corollary_trascendental}, for every $z$ in $\Sigma$ and $r\leq \rho_0$, the geodesic ball $\mathcal B_{r}(z)$ in $\Sigma$ satisfies 
$$ 
{\rm Area}(\mathcal B_{r}(z))\geq c r^{n-1}
$$
where $c$ is a constant depending only on $n$. A general result for Riemannian manifolds with boundary (see e.g. \cite[Proposition A.1]{INDIANA}) implies that there exists a piecewise geodesic path parametrized by arc length $\gamma\colon  [0,L]\to \Sigma_{\delta/2}$ connecting $p_0$ to $p$ and of length $L$ bounded by $\delta J_{\delta}$, where $J_\delta$ is given by \eqref{N_delta_hyp}.

	%
	We define $p_i=\gamma(r_i/4)$, for $i=1,\dots,J-1$  and $p_J=p$. Our choice of $r_0$ guarantees that $\mathcal U_{r_0}(p_i)\subset \Sigma$, for every $i = 0,\ldots, J$, and the other required properties are satisfied by construction.
\end{proof}

Since $p$ and $p_0$ are in a connected component of $\Sigma_\de$, there exists a sequence of points $p_1,\dots,p_J$ in the connected component of $\Sigma_{\delta/2}$ containing $p_0$, with $J\in \NN$ and $p_J=p$, and a chain of subsets $\{\mathcal U_{r_0}(p_i)\}_{\{i=0,\ldots,J\}}$ of $\Sigma$ as in lemma \ref{catenadipalle}.
We notice that $\Sigma$ and  $\hat\Sigma$ are tangent at
$p_0$ and that in particular the two normal vectors to $\Sigma$ and $\hat \Sigma$ at $p_0$ coincide. 
Now we apply the map $\varphi_{p_0}$ (see section \ref{section3}). Then  $\varphi_{p_0}(\Sigma)$ and $\varphi_{p_0}(\hat\Sigma)$  can be locally parametrized near $\varphi_{p_0}(p_0)$ as graphs of two functions $u_0,\, \hat u_0\colon B_{r_0} \subset \{x_n=0\} \to \RR$. Lemma \ref{Lemma1} implies that $| \nabla u_0|, |\nabla \hat u_0 | \leq M$ in $B_{r_0}$, where $M$ is some constant which depends only on $r_0$, i.e. only on $\rho$.  Hence the difference $u_0-\hat u_0$ solves a second-order linear uniformly elliptic equation of the form
$$
\mathcal L(u_0-\hat u_0)(x)={\sf H}_{\Gamma(u_0)}(x,u_0(x))-{\sf H}_{\Gamma(\hat u_0)}(x,\hat u_0(x))
$$
with ellipticity constants uniformly bounded by a constant depending only on $n$ and $\rho$. Since  $u_0(0)=\hat u_0(0)$ and $u_0 \geq \hat u_0$,
Harnack's inequality (see Theorems 8.17 and 8.18 in \cite{GT}) yields  
$$
\sup_{B_{r_0/2}} (u_0-\hat u_0)\leq C\,{\rm osc} ({\sf H}_S) \,,
$$
and from interior regularity estimates (see e.g. \cite[Theorem 8.32]{GT}) we obtain 
\begin{equation}\label{harnack step 1}
\|u_0- \hat u_0\|_{C^1(B_{r_0/4})} \leq C\,{\rm osc} ({\sf H}_S),
\end{equation}
where $C$ depends only on $\rho$ and $n$. Now we use lemmas \ref{lem_p_pstar1} and \ref{lem_p_pstar2}. Since $p_1\in \partial\, \mathcal U_{r_0/4}
(p_0)$, we can write $\varphi_{p_0}
(p_1)=(x_1,u_0(x_1))$, with $x_1\in \partial B_{r_0/4}$. Let $\hat p_1^* \in \hat\Sigma$ be such 
that 
$$
\varphi_{p_0}(\hat p_1^*)=(x_1,\hat u(x_1))
$$
and let $\hat p_1$ be the first intersection point between $\hat \Sigma$ and the geodesic path $\gamma$ starting from $p_1$ and 
tangent to $-N_{p_1}$ at $p_1$. From \eqref{harnack step 1} we have
\begin{equation} \label{conad}
d(\varphi_{p_0}(p_1),\varphi_{p_0}(\hat p_1^*)) + |\nu_{\varphi_{p_0}(p_1)} - \nu_{\varphi_{p_0}(\hat p_1^*)} | \leq C\,{\rm osc} ({\sf H}_S)
\end{equation}
%
which implies that  the assumptions in lemmas \ref{lem_p_pstar1} and \ref{lem_p_pstar2} are fullfilled, and we obtain
\begin{equation} \label{esselunga}
d(p_1,\hat p_1) + |N_{p_1} - \tau_{\hat p_1}^{p_1}(N_{\hat p_1})|_{p_1} \leq C\,{\rm osc} ({\sf H}_S) \,,
\end{equation}
where $C$ depends only on $n$ and $\rho$.

Now we apply $\varphi_{p_1}$. By definition of $\varphi_{p_1}$, we have  $\varphi_{p_1}(\hat p_1)=te_n$ for some $t\in\RR$ ($t$ depends on the geometry of the ambient space). A standard computation yields   
$$
|\nu_{\varphi_{p_1}(p_1)}-\nu_{\varphi_{p_1}(\hat p_1)}|= |N_{p_1} - \tau_{\hat p_1}^{p_1}(N_{\hat p_1})|_{p_1} \,
$$
which in view of \eqref{esselunga} implies 
$$
|\nu_{p_1}-\nu_{\hat p_1}| \leq C \osc({\sf H}_S) \,,
$$
where $C$ is a constant that depends only on $\rho$ and $n$. 
Since $\osc({\sf H}_S) \leq \ep$  then $|\nu_{p_1}-\nu_{\hat p_1}|<C\ep$ and by choosing $\ep$ such that  $C\ep < 1$, we can apply \cite[Lemma 3.4]{JEMS} and obtain that $\Sigma$ and $\hat \Sigma$ are locally graphs of two functions
$$
u_1,\hat u_1: B_{r_1} \to \RR^+ \,,
$$
such that $u_1(0)=\varphi_{p_1}(p_1)$ and $\hat u_1(0)= \varphi_{p_1}(\hat p_1)$ and where
$$
r_1=(1-C^2\ep^2)r<\rho_1 \,.
$$ 
Now, we can iterate the argument we did before. Indeed, since 
$$
0 \leq \inf_{B_{r_1/2}} (u_1-\hat u_1) \leq u_1(0)-\hat u_1(0) \leq C \osc({\sf H}_S)\,,
$$
by applying Harnack's inequality we obtain that
$$
\sup_{B_{r_1/2}} (u_1-\hat u_1)\leq C\,\osc({\sf H}_S)
$$
and from interior regularity estimates we find 
\begin{equation}\label{harnack step 2}
\|u_1- \hat u_1\|_{C^1(B_{r_1/4})} \leq C \,\osc({\sf H}_S)\,,
\end{equation}
where $C$ depends only on $\rho$ and $n$. Hence, \eqref{harnack step 2} is the analogue of \eqref{harnack step 1}, and we can iterate the argument. The iteration goes on until we arrive at $p_N=p$ and obtain a point $\hat p_N \in \hat \Sigma$ such that
$$
d(p,\hat p_N) + |N_p-\tau_{\hat p_N}^{p}(N_{\hat p_N})|_p \leq C\osc({\sf H}_S) \,.
$$
In view of lemma \ref{lemma connected II} we have that $\Sigma_\delta$ is connected and the claim follows.

\subsubsection{Case 2: $d_\Sigma(p_0,\pa \Sigma) \geq \de$ and $d_\Sigma(p,\pa \Sigma) < \de$} \label{subsec case2}
We extend the estimates found in case 1 to a point $p$ which is far less than $\delta$ from the boundary of $\Sigma$. Let $q\in \Sigma$ and $p_{min}\in \partial \Sigma$ be such that
$$
d_{\Sigma}(q,\partial \Sigma)=\delta \,, \quad d_{\Sigma}(p,q)+d_{\Sigma}(p,\partial \Sigma )=\delta\,, \quad  \mbox{and }\quad  d_{\Sigma}(p,p_{min})=d_{\Sigma}(p,\partial \Sigma)\,.
$$
From case 1 we have that there exists $\hat q$ in $\hat \Sigma$ such that
$$
d(q,\hat q)+|N_q-\tau_{\hat q}^q(N_{\hat q})|_q \leq C\,{\rm osc}({\sf H}_S)\,.
$$
Lemma \ref{lemma connected II} (case $(ii)$) yields that 
\begin{equation} \label{bellachegirano}
0 \leq g_z(N_z,\omega_z) \leq \tfrac{1}{4},
\end{equation}
for any $z \in \partial \Sigma$ and $\Sigma_\delta$ is connected.

Let $q^{\pi}\in S$ be the reflection of $q$ about $\pi$  and fix  $r\leq \rho_1$ in order to define $\mathcal U_r(q^\pi)$. We denote by $U_{r}(q)$ the reflection of
$\mathcal U_r(q^\pi)\cap S$ about $\pi$ and $U'=\mathcal{U}_r(q^\pi)\cap \pi$. Proposition \ref{prop Luigi I} implies that $U'$ is a hypersurface of $\pi$ with an induced orientation and its principal curvatures $\kappa_i'$ satisfy the following bounds 
\begin{equation*}
\frac{1}{\sqrt{1-g_z(N_z,\omega_z)^2}}\kappa_1(z)\leq \kappa'_i(z)\leq \frac{1}{ \sqrt{1-g_z(N_z,\omega_z)^2} }\kappa_{n-1}(z) \,,
\end{equation*}
for every $z\in U'$ and $i=1,\dots,n-1$. From \eqref{bellachegirano} and since $|\kappa_i(z)| \leq \rho^{-1}$ for any $z \in S$ (this follows from the touching ball condition), we have
\begin{equation}\label{nu}
|\kappa'_i(z)|\leq \frac{2}{\rho} \,,
\end{equation}
for any $z	\in U'$. 
Let $U''$ be the Euclidean orthogonal projection of $\varphi_q(U')$ onto $\{x_{n}=0\}$.  In  order to apply Carleson estimates in \cite[Theorem 1.3]{BCN}, 
we need to prove the following 
\begin{lemma}
Let $\{\kappa''_1,\dots, \kappa''_{n-2}\}$ be the Euclidean principal curvature of $U''$ viewed as a hypersurface of $\RR^{n-1}$. Then 
\begin{equation}\label{lucionediavolo}
\|\kappa_i''\|_{\infty} \leq C\,,\quad i=1,\dots,n-2\,,
\end{equation}
for some constant $C=C(\rho)$. 
\end{lemma}
\begin{proof}
We refer to \cite{JEMS} and \cite{INDIANA} for the proof of the assertion in the Euclidean and the Hyperbolic spaces, respectively, and we focus here on the spherical case.  
Here we use the same notation as in section \ref{subsect Luigi}. In particular we recall that for $z\in \varphi_q( U' )$, $\bar z$ is the projection of $z$ onto $U''$. 
Proposition \ref{prop Luigi II} yields 
$$
|\kappa_i''(\bar z)|\leq \frac{4 |(z-O_{\pi})_n|}{R(1+|z|^2)^2}\left((\nu_z'\cdot e_{n})^2+\frac{(z-O_{\pi})_n^2}{R^2}\right)^{-3/2}
\left(\max\{|\kappa_1'(z)|,|\kappa_{n-1}'(z)|\}+4|z|\right)
$$
for every $z\in U'$, where $O_{\pi}$ is the center of $\varphi_{q}(\pi)$ in $\RR^n$. 

We notice that the radius of $\varphi_{q}(\pi)$ is given by 
$$
R=\tfrac{1}{2a}+\tfrac{a}{2} \,,
$$
where $a$ is the Euclidean distance between $\varphi_q(\pi)$ and the origin of $\RR^n$. This follows from the proof of lemma \ref{lem_p_pstar2}: indeed, up to apply a rotation, we may assume that the point on $\pi$ having minimal Euclidean distance from the origin is $(a,0,\dots,0)$, with $a>0$, and that the normal to $\pi$ in $(a,0,\dots,0)$ is $e_1$. From a straightforward calculation we obtain the value of $R$.

Since $d(q,\pi)<\delta$ we have  $a<\delta $ which implies 
\begin{equation}\label{stimaR}
R>\frac{1}{2\delta}\,. 
\end{equation}
Since $U' \subset \mathcal{U}_r(q)$ and $\varphi_q(q)=O$, our choice of $r$ implies that for any  $z\in \varphi_q(U')$  we have
$$
|z|\leq \rho_1=\frac{\rho}{\pi} 
$$
and 
$$
|z-O_{\pi}|\leq |z-q|+|q-O_{\pi}|\leq \delta +R\leq 2R\,,
$$
and then \eqref{stimaR} implies  
\begin{equation}
\begin{aligned}
|\kappa_i''(\bar z)|
&\, \leq 8 (\nu_z'\cdot e_{n})^{-3}
\left(\max\{|\kappa_1'(z)|,|\kappa_{n-1}'(z)|\}+4\rho_1 \right)\,. 
\end{aligned}
\end{equation}

Now we give a lower bound of $\nu_z'\cdot e_{n}$. We can write 
$$
\nu_z'\cdot e_{n} = \nu_z'\cdot (e_{n} - \nu_z) + \nu'_z \cdot \nu_z = \nu_z'\cdot (e_{n} - \nu_z) + g_z(N'_z,N_z)\,.
$$
We firstly observe that
\begin{equation}\label{firstly}
g_z(N'_z,N_z)\geq \frac12 \,.
\end{equation} 
Indeed formula \eqref{*formula} yields 
$$
g_z(N'_z,N_z)=\sqrt{1-g_z(\omega_z, N_z)^2}\,,
$$
and  \eqref{bellachegira} implies \eqref{firstly} (here we can change configuration by considering the stereographic projection from the antipodal point to $\varphi_q^{-1}(z)$ in order to regard $\pi$ as a vertical hyperplane of $\mathbb{R}^n$). Moreover, from lemma 2.1 in \cite{JEMS} we can choose $r$ small enough in terms of $\rho_1$ in order to obtain $|e_{n} - \nu_z| \leq 1/4$.  Hence 
$$
\nu_z'\cdot e_{n} = \nu_z'\cdot (e_{n} - \nu_z) + \nu'_z \cdot \nu_z \geq \frac12-|\nu_z'\cdot (e_{n} - \nu_z)|\geq \frac14
$$
and 
$$
|\kappa_i''(\bar z)| \leq C\,,
$$
for some constant $C=C(\rho)$, as required.
\end{proof}

We denote by $x$ and $E_r$ the projections of $\varphi_q(p_{min})$ and $\varphi_q(U_r(q))$ onto $\{x_n=0\}$, 
respectively. The Euclidean distance of $x$ from $\partial E_r$ is less than $C\delta$ where $C$ depends only on $\rho$ and, up to chose a smaller $\delta$ in terms of $\rho$, the projection of $p_{min}$ stays close to $U''\subset\partial E_r$ and we can apply theorem 1.3 in \cite{BCN}, corollary 8.36 in \cite{GT} and Harnack's inequality (see e.g. \cite[Corollary 8.36]{GT}) to obtain
\begin{equation}\label{lorenzino}
\sup_{B_{2\delta C}(x)\cap E_r}(u-\hat{u})\leq C (u-\hat{u})(z)+\osc({\sf H}_S)
\end{equation}
with $z=x+4C\delta\nu''_x$, where $\nu''_x$ is the interior normal to $U''$ at $x$. Thanks to \eqref{lucionediavolo} and by choosing $\delta$ small enough in terms of $\rho$, from \eqref{lorenzino} and Harnack's inequality we obtain 
\begin{equation} \label{violaharnack}
0 \leq  \|u - \hat u\|_{C^1(B_{C\de}(x) \cap E_r)} \leq C ( (u(0) - \hat u (0))  + \osc({\sf H}_S) )\,.
\end{equation}
Since $d_\Sigma(q, \partial \Sigma)=\delta$, from Case 1 we know that 
\begin{equation*}
d(q,\hat q) + |N_{q} - \tau_{\hat q}^q(N_{\hat q}) |_q \leq C\osc({\sf H}_S) \,,
\end{equation*}
and from \eqref{violaharnack} we obtain that 
\begin{equation} \label{violaharnack_I}
0 \leq  \|u - \hat u\|_{C^1(B_{C\de}(x) \cap E_r)} \leq C \osc({\sf H}_S) \,.
\end{equation}
From lemma \ref{lem_p_pstar2} we deduce
$$
d(p,\hat p) + |N_{p} - \tau_{\hat p}^p (N_{\hat p}) |_{p} \leq C \osc({\sf H}_S) \,,
$$ 
as required.

\subsubsection{Case 3: $0 < d_\Sigma(p_0,\pa \Sigma) < \de$.} We first prove the following preliminary lemma which implies via lemma 	\ref{lemma connected} that $\Sigma$ is connected.

\begin{lemma}
By choosing $\delta$ small enough in terms of $\rho$, the following inequality holds 
\begin{equation} \label{quasi_ortog_p0}
0 \leq  g_{p_0}(N_{p_0},\omega_{p_0}) \leq \frac{1}{4} \,.
\end{equation}
\end{lemma}
\begin{proof}
We assume the statement in the Euclidean case (see  \cite[Section 4.1.3]{JEMS}) and we show how to deduce the 
claim in the hyperbolic and in the spherical case. 
We first consider the Hyperbolic case.  Up to apply an isometry we can assume that $p_0=e_n$ and $\pi=\{x_1=0\}$. 
Our assumptions on $S$ imply that its diameter is bounded in terms of $\rho$ and $|S|_g$ (see e.g. \cite[Proposition A.2]{INDIANA}).  Therefore $S$ is 
contained in an Euclidean ball about the origin and of radius depending only on $\rho$ and $|S|_g$. Up to choose $\delta$ 
small enough in terms of $\rho$, we have that \cite[Section 4.1.3]{JEMS} implies 
$$
0 \leq \nu_{p_0}\cdot e_1 \leq \frac{1}{4} \,.
$$
Since $\nu_{p_0}\cdot e_1= g_{p_0}(N_{p_0},\omega_{p_0})$, the claim follows.  
In the spherical case the proof is analogue once the setting is modified as follows: we work in $(\RR^n,g)$, where $g$ is the round metric \eqref{round metric},  assuming that $p_0=O$ and $\pi$ is an Euclidean 
hyperplane. 
\end{proof}
%
%
%
%
%
%
Then we prove the existence of a point $q\in \Sigma$ such that 
\begin{equation}\label{martedi}
\begin{cases}
d(q,\hat q) + |N_q - \tau_{\hat q}^q (N_{\hat q})|_q  \leq C \osc({\sf H}_S)\\
d_\Sigma (q,\pa \Sigma) \geq \delta
\end{cases}
\end{equation}
and we apply cases 1 and 2 to conclude.

In the same fashion as in case 2, we can locally write 
$\varphi_{p_0}(\Sigma)$ and $\varphi_{p_0}(\hat \Sigma)$ as graphs of function $u, \hat u\colon E_r\to \RR$ near $\varphi_{p_0}(p_0)$, respectively. 
Without loss of generality we can assume $r<1$ (indeed $r$ must be chosen small enough in terms of $\delta$). Let $U''\subset\pa E_r$ be
the projection of $\varphi_{p_0} ({U}_r(p_0) \cap \pi)$ onto $\lbrace x_n=0\rbrace$. 
Analogously to case 2,  the Euclidean principal curvatures of $U''$ are bounded by a constant 
$\mathcal{K}$ depending only on $\rho$. Then we can argue as in \cite[Section 4.1.3]{JEMS} and find a point $y \in E_r$ of the form 
$$
\begin{aligned}
y=\bar x +2 c_*\de \nu''_{\bar x} 
\end{aligned}
$$
where $c_*$ is $1$ in the Euclidean and in the spherical case, while it is  the constant $c$ appearing in \eqref{oggilunedi} in the hyperbolic case and  
 $\bar x \in U''$ is such that
$$
|\bar x| = \min_{x \in U''} |x| \,.
$$
By choosing $\delta$ sufficiently small in terms of $\rho$ we have 
$$
d(w,\hat w^*) + |\nu_w - \nu_{\hat w^*}| \leq C \osc({\sf H}_S) \,,
$$
where $w=(y,u(y))$, $\hat w^*=(y,\hat u(y))$. Lemmas \ref{lem_p_pstar1} and \ref{lem_p_pstar2} yield 
$$
d(q,\hat q) + |N_q - \tau_{\hat q}^q (N_{\hat q})|_q  \leq C \osc({\sf H}_S) \,,
$$
where $q=\varphi^{-1}_{p_0}(w)$, $\hat q^*=\varphi^{-1}_{p_0}(\hat w^*)$ and $\hat q$ the first intersection point between  $\hat \Sigma$ and the geodesic path starting from $q$ and tangent to 
$-N_{q}$ at $q$.

Let $z$ be a point on $\partial U_r(p_0)$ realizing $d(q,\partial U_r(p_0))$.  By construction and from lemma \ref{distance} we have 
$$
d_\Sigma (q,\pa \Sigma) \geq  d(q,z) \geq  c_* |\varphi_{p_0}(q)-\varphi_{p_0}(z)|\geq 2\delta \,.
$$
Since $d_\Sigma (q,\pa \Sigma) \geq \delta$ $q$ satisfies \eqref{martedi} and the claim follows.

\begin{figure}[h]
	\centering
	\includegraphics[scale=1]{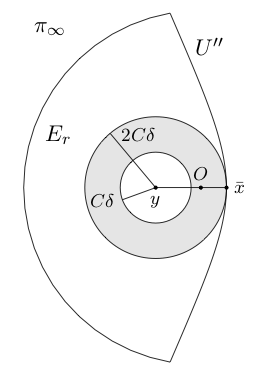}
	\caption{Case 3 in the proof of theorem \ref{thm approx symmetry 1 direction}.}
	\label{fig case 3}
\end{figure}

\subsubsection{Case 4: $p_0 \in \pa \Sigma$.}
This is the limit configuration of case 3 when $d_\Sigma(p_0,\pa \Sigma) \to 0$. Indeed, here $E_r$ is a half-ball in $\mathbb{R}^n$ and the argument used in case $3$ can be easily adapted. This completes the proof of the first part of theorem \ref{thm approx symmetry 1 direction}.

\subsubsection{Last step:  $d(x,\Sigma \cup \Sigma^\pi)\leq  C \osc({\sf H}_S) $ for every $x\in \Omega$.}
Assume by contradiction that 
$$
d(x,\Sigma \cup \Sigma^\pi)> C \osc({\sf H}_S)
$$ 
for some $x$ in $\Omega$. Since $\Omega$ is connected, it is possible to find $y\in \Omega$, such that
$$
y\in \Omega_-\quad \mbox{ and } \quad C \osc({\sf H}_S) < d(y,\Sigma) \leq 2 C \osc({\sf H}_S) \,,
$$
where 
$$
\begin{aligned}
&\Omega_+=\{p\in \Omega\,\,:\,\, p\in \pi_{v,t} \mbox{ for some }t>m_v\}\,,\\
&\Omega_-=\{p\in \Omega\,\,:\,\, p\in \pi_{v,t} \mbox{ for some }t<m_v\}\,. 
\end{aligned}
$$
Let $p$ be a projection of $y$ over $\Sigma\cup \Sigma^{\pi}$. If $p\in\Omega_+$, then $y$ belongs to the exterior touching ball of $S$ at $p$, which gives a contradiction. The same contradiction is obtained when  $p\in\pi$ since, in that case $g_p(N_p, \omega_p)\leq 1/4$. If $p\in \Omega_-$, we can find a point $\hat p\in S$ such that $\hat p$ and $p$ lies on the geodesic $\gamma$ starting from $p$ and orthogonal to $\Sigma^\pi$ and such that
$$
d(p,\hat p) + |N_p - \tau_{\hat p}^p (N_{\hat p})|  \leq C \osc({\sf H}_S) \,.
$$
By the smalleness of $\osc({\sf H}_S)$ we obtain that $y$ belongs to the exterior touching ball of $S$ at $p$, which is a contradiction.

 \hspace*{\fill}  \begin{math}\Box\end{math}

\section{ Global approximate symmetry \label{section_6}}
From the previous section we have that if a $C^2$-regular closed hypersurface  $S= \partial \Omega$  embedded in $\mathbb M^n_+$ satisfies the assumptions of theorem \ref{thm approx symmetry 1 direction} then it is almost symmetric with respect to any direction, with the almost symmetry quantified by the deficit ${\rm osc }({\sf H}_S)$. In this section we show how this result leads to the almost radial symmetry of $S$. Such procedure is not peculiar of the kind of deficit considered, but it can be applied whenever one has the approximate symmetry in any direction with respect to some deficit.  More precisely we consider the following 

\begin{definition}
{\rm Let  $\Upsilon$ be the space of open sets $\Omega$ in $\mathbb M^n_+$ whose  
boundary is a $C^2$-regular connected closed  embedded  hypersurface, with the topology induced by the Hausdorff distance. A {\em deficit function} is any continuous function $\deficit: \Upsilon \to [0,+\infty)$ such that $\deficit (\Omega) = 0$ if and only if $\Omega$ is a ball.}
\end{definition}

Form now on we fix a deficit function ${\rm def}.$ 
\begin{definition}
{\rm We say that a bounded open set $\Omega$ satisfies the \emph{approximate symmetry property (ASP)} if there exists a constant $\mathcal K>0$ satisfying the following condition: for every direction $v$ there exists a connected component $\Sigma$ of the maximal cap in the direction $v$ such that 
$$
d(p,\Sigma \cup \Sigma^{\pi_v}) \leq \mathcal K\,\deficit(\Omega) \,,
$$
for every $p \in \Omega$.}
\end{definition}

The main theorem in this section is the following

\begin{theorem} \label{thm_novaga}
Let $S=\partial \Omega$ be a $C^2$-regular closed hypersurface embedded in $\mathbb M^n_+$, with $\Omega$ satisfying (ASP) and 
\begin{equation} \label{luigione0}
\deficit(\Omega) \leq \frac{|\Omega|_g}{4 \mathcal{K}}  \,.
\end{equation} 
There exist $\mathcal O  $ in $\mathbb M_+^n$ and two balls $B_r^d$ and $B_R^d$ centered at $\mathcal{O}$ of radius $r$ and $R$, respectively, with $r\leq R$, such that 
\begin{equation*}
B^d_r(\mathcal{O}) \subseteq \Omega \subseteq B^d_R(\mathcal{O})
\end{equation*}
and 
\begin{equation}\label{11}
R-r \leq C \, \deficit(\Omega) \,,
\end{equation}
where $C$ depends on $n, \rho ,|S|_g$ and $\mathcal{K}$.
\end{theorem}

The following lemma is needed in order to  prove theorem \ref{thm_novaga}.

\begin{lemma} \label{lemmarta}
Let $S=\partial \Omega$ be a $C^2$-regular closed hypersurface embedded in $\mathbb M^n_+$, with $\Omega$ satisfying (ASP) and \eqref{luigione0}. Then there exists $\mathcal O  $ in $\mathbb M_+^n$ such that 
$$
d(\mathcal O, \pi_{v}) \leq C  \deficit(\Omega)\,,
$$
for every direction $v$ in  $T_{\origin}\mathbb M^n$, where $C$ depends on $n, \rho ,|S|_g$ and $\mathcal{K}$.
\end{lemma}

\begin{proof} 
We fix an orthonormal basis $\{e_1,\dots,e_n\}$ of the tangent space at the ``origin'' $\origin$ and we consider the corresponding critical hyperplanes $\pi_{e_i}$.  We define an approximate center of symmetry $\mathcal{O}$ as follows: 
$$
\mathcal O:= \bigcap_{i=1}^n \pi_{e_i} \,.
$$
We notice that in the Euclidean case $\mathcal O$ is well-defined. Although in the hyperbolic space $n$ orthogonal hyperplanes do not always intersect, we can work as in 
\cite{INDIANA}[Lemma 6.1] and showing that \eqref{luigione0} implies the existence of $\mathcal{O}$. In $\mathbb S^{n}_+$ the existence of $\mathcal{O}$ is always guaranteed. Indeed
every $\pi_{e_i}$ is given by the intersection of a plane $\Pi_{e_i}$ of $\RR^{n+1}$ with $\mathbb S^{n}_+$ and the intersection of all the $\Pi_{e_i}$'s is a straight line $r$ which, by construction, can not lie in the plane $\{x_{n+1}=0\}$; hence  $\mathcal O=r\cap \mathbb S^{n}_+ \neq \O$ . 

Let $\mathcal{R}$ be the reflection about $\mathcal O$. Note that 
$$
\mathcal{R}(p) = \pi_{e_1} \circ \cdots \circ \pi_{e_n} (p)\,,
$$
where we identify $\pi_{e_i}$ with the reflection about the corresponding hyperplane.

Let $v$ be a fixed direction and let $\Omega_v$ be the corresponding maximal cap. Since $\Omega$ satisfies $(ASP)$, we have
\begin{equation} \label{aereo}
|\Omega_v|_g \geq \frac{|\Omega|_g}{2} - C \deficit(\Omega) \,,
\end{equation}
for some constant $C$ depending only on $n,\rho,|S|_g$ and $\mathcal{K}$. Moreover we have  
\begin{equation} \label{aereo2}
|\Omega \triangle \Omega^\pi |_g = 2 (|\Omega|_g - 2 |\Omega_v|_g) \leq 4C \deficit(\Omega) \,,
\end{equation}
where $\Omega \triangle \Omega^\pi $ denotes the symmetric difference between $\Omega$ and $\Omega^\pi$. Next we work as in lemma 4.1 in \cite{CFMN}. 
Here we only sketch the argument referring to \cite{CFMN} for details (see also \cite{JEMS}). 

Without lost of generalities, we may assume $\mathcal O\in \pi_{v,m_v-\mu}$, for some $\mu>0$ and for $k\in \mathbb N$ we define 
$$
\mu_k= \big| \{ p \in  \Omega \cap \pi_{v,s}:\ m_v + (k-1)\mu <s<m_{v} + k \mu \} \big|_g\,.
$$
By construction $\mu_k$ is decreasing and, in particular,
$$
\mu_k\leq \mu_0:= \big| \{ \Omega \cap \pi_{v,s}:\  m_v-\mu <s<m_{v}  \} \big|_g \,.
$$   
Moreover, $\mu_0$ is bounded by $C\deficit(\Omega)$. Indeed, formula \eqref{aereo} yields 
$$
|\Omega \triangle \mathcal{R}(\Omega)|_g \leq C \deficit(\Omega) \,,
$$
and then we obtain
$$
|\Omega \cap \mathcal{R} (\Omega_{v}) |_g\geq |\Omega_{v}|_g - |\Omega \triangle \mathcal{R}( \Omega)|_g \geq \frac{|\Omega|_g}{2} - C \deficit(\Omega) \,.
$$
Since   
$$
\mathcal{R}(\Omega_{v}) \subset \bigcup_{s<0} \pi_{v,m_{v}-s} \,,
$$
we obtain that 
\begin{equation*} 
\mu_0:= \big| \{ \Omega \cap \pi_{v,s}:\  m_v-\mu <s<m_{v}  \} \big|_g \leq C \deficit(\Omega) \,.
\end{equation*}
Therefore 
\begin{equation}\label{indiana}
\mu_k\leq C\deficit(\Omega)
\end{equation}
for every $k$ in $\NN$. From \eqref{aereo} we get 
$$
\frac{|\Omega|_g}{2} - C\deficit(\Omega) \leq |\Omega_v |_g\leq \sum_{k=0}^{k_0} \mu_k  \leq k_0 \mu_0 \leq \frac{\diam(\Omega)}{m_v} C \deficit(\Omega)
$$
where $k_0$ is the integer part of $\tfrac{\diam(\Omega)}{m_v}$.  
From Proposition A.1 in \cite{INDIANA} we have
$$
m_v \leq C \deficit(\Omega)\,,
$$
where $C$ depends only on $n$, $\rho,|S|_g$ and $\mathcal{K}$, as required. 
\end{proof}

\begin{proof}[Proof of theorem \ref{thm_novaga}]
Let $\mathcal{O}$ be as in lemma \ref{lemmarta} and define
$$
r=\sup\{s>0:\ B^d_s(\mathcal O)\subset \Omega\} \quad \textmd{ and } \quad R=\inf\{s>0:\ B^d_s(\mathcal O)\supset \Omega\} \,,
$$
so that $B^d_r(\mathcal{O}) \subseteq \Omega \subseteq B^d_R(\mathcal{O})$.

Let $p,q \in S$ be such that $d(p,\mathcal O)=r$ and $d(q,\mathcal O)=R$. We can assume that $p \neq q$ (otherwise $r=R$ and $S$ is a round sphere). Let $v\in T_{\origin} \mathbb M^n$ be the direction 
$$
v:=\frac{1}{d(p,q)}\tau_{p}^{\origin}(\exp_p^{-1}(q)) 
$$
and $\pi_v$ the critical hyperplane in the $v$-direction. We denote by $\gamma$ the geodesic path passing through $p$ and $q$ and let $s_p$ and $s_q$ in $\RR$ be such that 
$$
\gamma(s_p)=p \mbox{ and } \gamma(s_q)=q\, .
$$
Let  $z\in \pi_v$ be such that $d(z,\mathcal O)=d(\mathcal O,\pi_v)$. We have 
$$
p\in \pi_{v,s_p}\,,\quad q\in \pi_{v,s_q}\,,\quad s_q=s_p+t\,;
$$
see section \ref{moving} for the definition of $\pi_{v,s_p}$ and $\pi_{v,s_q}$.
We first show that $d(q,z)\leq d(p,z)$. Assume by contradiction that $d(q,z) > d(p,z)$. Since $q$ and $p$ belong to a geodesic orthogonal to the hyperplanes $\pi_{v,s}$ and $s_p< s_q$, then $s_q>m_v$. Since $\pi_v=\pi_{v,m_v}$ corresponds to the critical position of the method of the moving planes in the direction $v$, we have that $\gamma(s) \in \Omega$ for any $s \in (m_v,s_q)$. Since $s_p<s_q$ we have that $|s_p-m_v| \geq |s_q-m_v|$ and since $\gamma$ is orthogonal to $\pi_v$ we obtain $d(q,z)\leq d(p,z)$, which gives a contradiction. 
Since $d(q,z)\leq d(p,z)$ we have
$$
r \geq R - d(\mathcal O, z) = R - d(\mathcal O, \pi_v) 
$$
and lemma \ref{lemmarta} implies \eqref{11}.
\end{proof}

%
\section{Proof of the main results} \label{section_proof_main2}
We have all the ingredients to prove theorem \ref{main} and corollary \ref{main2}. 

\begin{proof}[Proof of theorem \ref{main}]
Let $S=\partial \Omega$ be a $C^2$-regular, connected, closed hypersurface embedded in $\mathbb M^n_{+}$ satisfying a uniform touching ball condition of radius $\rho$, 
where $\Omega$ is a relatively compact domain. Theorem \ref{thm approx symmetry 1 direction} implies that there exit $\epsilon$ and $C$ positive such that if 
$$
{\osc}({\sf H}_S) \leq \ep,
$$
then 
$$
d(p,\Sigma \cup \Sigma^{\pi_v}) \leq C{\rm osc}({\sf H}_S)\,,
$$
for every $p \in \Omega$.
\end{proof}

\begin{proof}[Proof of corollary \ref{main2}]
The proof consists in one more application of the method of the moving planes and it is in the spirit of \cite[Theorems 1.2 and 1.5]{CFMN}.
Let $B^d_{r}(\mathcal O)$ and $B^d_{R}(\mathcal O)$ be as in theorem \ref{thm_novaga} and let $0<t<r- C\deficit(\Omega)$. We aim at proving that for any $p\in S$, there exist two cones with vertex at $p$ and of fixed aperture, one contained in $\Omega$ and one contained in the complementary of $\Omega$. The first cone $C^-(p)$, is obtained by considering all the geodesic path connecting $p$ to the boundary of $B^d_{t}(\mathcal O)$ tangentially. The second cone $C^+(p)$ is the reflection of $\mathcal C^-(p)$ with respect to $p$. We show that $C^-(p)$ is contained in $\Omega$ and an analogous argument shows that $\mathcal C^+(p)$ is contained in the complementary of $\Omega$. We assume, by contradiction, that $p\notin B^d_{r}(\mathcal O)$ (otherwise the claim is trivial) and that there exists a point $q \in \mathcal C^-(p) \cap \partial B^d_{t}(\mathcal O)$ such that the geodesic path $\gamma$ connecting $q$ to $p$ is not contained in $\Omega$. We apply the method of the moving planes in the direction $v$ defined by
$$
v:=\frac{1}{d(p,q)}\tau_{q}^{\origin}(\exp_{q}^{-1}(p))\, .
$$
Since $\gamma$ is not contained in $\Omega$, the method of the moving planes stops before reaching $q$ and one can prove that
$$
d(\mathcal O, \pi_\omega)\geq r-t \,.
$$
Since $0<t<r- C\deficit(\Omega)$, from lemma \ref{lemmarta}, we obtain 
$$
C\deficit(\Omega)< r-t \leq d(\mathcal O, \pi_\omega)  \leq C\deficit(\Omega) \,,
$$
which gives a contradiction. The argument above shows also that for any $p \in S$ the geodesic path connecting $p$ to $\mathcal O$ is contained in $\Omega$. This implies that there exists a $C^2$-regular map $\Psi: \partial B_{r}^d(\mathcal O) \to \mathbb{R}$ such that
$$
F(p) = \exp_x(\Psi(p)N_p) \,,
$$
defines a $C^2$-diffeomorphism  from $B^d_r(\mathcal O)$ to $S$. By choosing $t=r - \sqrt{C\deficit(\Omega)}$ we have that for any $p \in S$ there exists a uniform cone of opening $\pi - \sqrt{C\deficit(\Omega)}$ with vertex at $p$ and axis on the geodesic connecting $p$ to $\mathcal O$. This implies that $\Psi$ is locally Lipschitz and the bound \eqref{fristLipbound} on $\|\Psi\|_{C^1}$ follows (see also \cite[Theorem 1.2]{CFMN}).
\end{proof}

\begin{remark}{\em
We observe that if ${\sf H}_S=H$ is the mean curvature of $\partial \Omega$,  then \eqref{fristLipbound} can be improved and we can obtain the optimal linear bound $\|\Psi\|_{C^{1,\alpha}}\leq C \oscH$ by using elliptic regularity.  
Indeed, let $\phi\colon U\to \partial B^d_r(\mathcal O)$ be a local parametrization of $\partial B^d_r(\mathcal O)$, where $U$ is an open set of $\RR^{n-1}$. From the proof of corollary \ref{main2}, $F\circ \phi$ gives a local parametrization of $S$. A standard computation yields that 
$$
L(\Psi\circ \phi)=H(F\circ\phi)-H_{\mathsf{B}_r}
$$
where $H_{\mathsf{B}_r}$ is the mean curvature of $\partial \mathsf{B}_r$ and $L$ is an elliptic operator which, thanks to the bounds on $\Psi$ above, can be seen as a second order linear operator acting on $\Psi\circ \phi$. Then \cite[Theorem 8.32]{GT} implies the bound on the $C^{1,\alpha}$-norm of $\Psi$, as required.}
\end{remark}

\medskip

\section*{List of symbols}

In this last section we collect some symbols we used in the paper.

\smallskip 
\noindent $\mathbb M^{n}$ denotes one of the following manifolds: the Euclidean space $\RR^n$, the hyperbolic space $\mathbb H^n$, the hemisphere $\mathbb S^n$.

\smallskip 
\noindent $\mathbb M^{n}_+$ denotes one of the following manifolds: the Euclidean space $\RR^n$, the hyperbolic space $\mathbb H^n$, the hemisphere $\mathbb S^n_+$.

\smallskip 
\noindent $S$  denotes a $C^2$-regular, connected, closed hypersurface embedded in $\mathbb M^n_{+}$, $\Omega \subset \mathbb M^n_+$ denotes a relatively compact connected open set such that $\partial\Omega=S$.

\smallskip 
\noindent $\kappa_1,\dots \kappa_{n-1}$ denote the principal curvatures of $S$ ordered increasingly.

\smallskip 
\noindent $H$ denotes the mean curvature of $S$.

\smallskip 
\noindent ${\sf H}_S$ denotes a more general function of the principal curvaures of $S$ (see the definition in the introduction).

\smallskip 
\noindent ${\osc}({\sf H}_S)$ denotes the oscillation of ${\sf H}_S$ on $S$.

\smallskip 
\noindent $\exp_x$ denotes the exponential map of a generic Riemannian manifold $(M,g)$ at $x\in M$.  

\smallskip 
\noindent ${\rm inj}$ denotes the injectivity radius.

\smallskip 
\noindent $\mathcal O$ denotes either the center of mass (in section \ref{moving}) or the approximate center of mass (in all other sections).

\smallskip 
\noindent $O$ denotes either the origin of $T_{p}S$ (in section \ref{moving}) or the origin in $\mathbb{R}^{n-1}$ (in all other sections).

\smallskip 
\noindent $\origin$ denotes the origin in $\RR^n$, $e_n$
in $\mathbb H^n$ and the north pole in $\mathbb S^n_+$.

\smallskip 
\noindent $T_{p}M$ dentoes the tangent space to a generic Riemannian manifold $(M,g)$ at $p\in M$ and $T_p S$ denotes the tangent hyperplane to $S$ at $p$.

\smallskip 
\noindent $d$ denotes the geodesic distance in $\mathbb M^{n}$ induced by the Riemannian metric and $d_S$ denotes the distance in $S$.

\smallskip 
\noindent $B_r$ denotes either a Euclidean ball of radius $r$ centred at the origin $O$ of $T_{p}S$ (in section \ref{moving}) or a Euclidean ball of radius $r$ centred at the origin $O$ of $\mathbb{R}^{n-1}$ (in all other sections).

\smallskip 
\noindent $B^d_{r}(p)$ denotes the ball, with respect to $d$, of radius $r$ in $\mathbb M^{n}$ centred at $p\in\mathbb M^{n}$.

\smallskip 
\noindent $\mathcal B_{r}(p)$ denotes the geodesic ball of $S$ of radius $r$ centred at $p\in S$.


\smallskip 
\noindent $\varphi_{p}\colon \mathbb M^{n}_+\to \RR^n$, for $p\in S$, denotes the following function whose definition depends on the geometry of $\mathbb M^n$: 
\begin{itemize}
	\item if $\mathbb M^n$ is $\RR^n$, $\varphi_{p}\in {\rm SO}(n)\rtimes \RR^n$ and it is such that $\varphi_p(p)=0$ and $\varphi_{p*|p}\left(T_pS\right)=\{x_n=0\}$;
	
	\vspace{0.1cm}
	\item if $\mathbb M^n$ is $\mathbb H^n$, $\varphi_p$ is an orientation preserving isometry of $\mathbb H^n$ such that  $\varphi_{p}(p)=e_n$, $\varphi_{p*|p}\left(T_pS\right)=\{x_{n}=0\}$;
	
	\vspace{0.1cm}
	\item if $\mathbb M^n$ is $\mathbb S^n$, $\varphi_{p}$ is the stereographic projection form the antipodal point to $p$ restricted to $\mathbb S^n_+$ composed with a rotation of $\mathbb R^n$ in order to have $\varphi_{p*|p}\left(T_pS\right)=\{x_{n}=0\}$.   
\end{itemize}

\smallskip 
\noindent $\mathcal U_r(p)$ denotes the open neighborhood of $p$ in $S$ such that $\varphi_p\left(\mathcal U_r(p)\right)$ is the (Euclidean) graph of a $C^2$-function $u\colon B_r\to \RR$ defined in the ball of radius $r$ of $\RR^{n-1}$ centered at the origin.

\smallskip 
\noindent $\rho$ denotes the radius of the touching ball condition of $S$ and $\rho_1$ denotes the following quantity:
\begin{itemize}
	\item $\rho_1=\rho$, if $\mathbb M^n=\RR^n$;
	
	\vspace{0.1cm}
	\item $\rho_1=(1-{\rm e}^{-\rho}\sinh\rho){\rm e}^{-\rho}\sinh\rho$,  if $\mathbb M^n=\mathbb H^n$;
	
	\vspace{0.1cm}
	\item $\rho_1=\tfrac{\rho}{\pi}$, if $\mathbb M^n=\mathbb S^n$\,.
\end{itemize}

\smallskip 
\noindent $\tau_p^q$ denotes the parallel transport along the unique geodesic path connecting $p$ to $q$.

\smallskip 
\noindent $|\cdot|_p:=g_{p}(\cdot,\cdot)^{1/2}$ where $p\in \mathbb M^n$. 
	
\smallskip 
\noindent $|\cdot|$ denotes the Euclidean norm.	

\smallskip 
\noindent $|\cdot|_g$ denotes either the volume with respect to the Riemannian metric $g$ or the area with respect to the Riemannian metric $g$.

	
\smallskip 
\noindent $N$ denotes the inward unitary normal vector field on $S$.     

\smallskip 
\noindent $\nu$ denotes the Euclidean inward unitary normal vector field on $S$.

\smallskip 
\noindent $\pi$, $S_+$, $S_-$ and $p_0$ see section \ref{sect5}.

\smallskip 
\noindent $S_+^\pi$ denotes the reflection of $S_+$ with respect to $\pi$.

\smallskip 
\noindent $\Sigma$ and $\hat \Sigma$  denote the connected components of $S_+^\pi$ and $S_-$ containing $p_0$, respectively.


%
%

%
%

\end{document}